\documentclass[12pt,reqno]{amsart}
\usepackage{amsmath,amssymb,latexsym, amsthm, amscd, mathrsfs, stmaryrd} 
\usepackage[linktocpage=true]{hyperref}
\hypersetup{colorlinks,linkcolor=blue,urlcolor=cyan,citecolor=blue}

\usepackage[all]{xy}
\usepackage{hyperref}
\usepackage{color}
\newcommand{\nc}{\newcommand}
\nc{\browntext}[1]{\textcolor{brown}{#1}}
\nc{\greentext}[1]{\textcolor{green}{#1}}
\nc{\redtext}[1]{\textcolor{red}{#1}}
\nc{\bluetext}[1]{\textcolor{blue}{#1}}
\nc{\brown}[1]{\browntext{ #1}}
\nc{\green}[1]{\greentext{ #1}}
\nc{\red}[1]{\redtext{ #1}}
\nc{\blue}[1]{\bluetext{ #1}}
\nc{\zb}[1]{\redtext{From zb: #1}}

\def \tt{\textup{\texttt{t}}}
\newtheorem{theorem}{Theorem}  [section]
\newtheorem{corollary}[theorem]{Corollary}
\newtheorem{lemma}[theorem]{Lemma}

\newtheorem{proposition}[theorem]{Proposition}

\newtheorem{definition}[theorem]{Definition}

\theoremstyle{remark}
\newtheorem{remark}[theorem]{Remark}

\numberwithin{equation}{section}
\def\lr#1{\langle #1\rangle}

\def \n{\mathfrak{b}}

\def \ct{{\mathcal T}}

\def \bF{\mathbf{F}}

\numberwithin{equation}{section}

\def \diag{\mathrm{diag}}

\newcommand{\mbf}{\mathbf}

\newcommand{\ev}{\bar{0}}

\newcommand{\N}{\mathbb N}

\newcommand{\odd}{\bar{1}}

\newcommand{\ov}{\overline}
\newcommand{\qbinom}[2]{\begin{bmatrix} #1\\#2 \end{bmatrix} }
\newcommand{\Q}{\mathbb Q}

\newcommand{\U}{\mbf U}

\newcommand{\arxiv}[1]{\href{http://arxiv.org/abs/#1}{\tt arXiv:\nolinkurl{#1}}}

\newcommand{\Ui}{{\mbf U}^\imath}

\newcommand{\vs}{\varsigma}

\newcommand{\Z}{\mathbb Z}

\def \fg{\mathfrak{g}}
\def \bU{{\mathbf U}}
\newcommand{\tK}{\widetilde{K}}
\def \I{\mathbb{I}}
\def \bv{v}

\newcommand{\tUi}{\widetilde{{\mathbf U}}^\imath}

\newcommand{\tU}{\widetilde{\mathbf U}}

\newcommand{\tk}{\Bbbk}

\def \ff{B}
\def \ffy{y}

\def \bvs{{\boldsymbol{\varsigma}}}

\allowdisplaybreaks

\begin{document}

\title[Analogue of Feigin's map on  $\imath$quantum group of split type]
{Analogue of Feigin's map on $\imath$quantum group\\ of split type}
\author[Ming Lu]{Ming Lu}
\address{Department of Mathematics, Sichuan University, Chengdu 610064, P.R.China}
\email{luming@scu.edu.cn}

\author[Shiquan Ruan]{Shiquan Ruan}
\address{School of Mathematical Sciences,
Xiamen University, Xiamen 361005, P.R.China}
\email{sqruan@xmu.edu.cn}

\author[Haicheng Zhang]{Haicheng Zhang}
\address{School of Mathematical Sciences, Nanjing Normal University, Nanjing 210023, P.R.China.}
\email{zhanghc@njnu.edu.cn}




\begin{abstract}
The (universal) $\imath$quantum groups are as a vast generalization of (Drinfeld double) quantum groups.
We establish an algebra homomorphism from universal $\imath$quantum group of split type to a certain quantum torus, which can be viewed as an $\imath$analogue of Feigin's map on the quantum group.
\end{abstract}

\maketitle


\section{Introduction}

\subsection{Feigin's map and Hall algebras}
Let ${\mathbf U}^+$ be the positive part of the quantum group associated to a generalized Cartan matrix $C\in M_{\I\times\I}(\mathbb{Z})$, and $\mathbf{i}=(i_1i_2\cdots i_r)$ be a word made up of elements in $\I$.
The Feigin homomorphism $\bF_{\mathbf{i}}$ (also called Feigin's map) from ${\mathbf U}^+$ to a quantum torus $\mathcal{T}_\mathbf{i}$ is an algebra homomorphism, which was proposed by Feigin in 1992 as an elementary tool for studying the skew-field of fractions of ${\mathbf U}^+$. Feigin gave a conjecture that $\bF_{\mathbf{i}}$ is injective if $C$ is of finite type and $\mathbf{i}$ is a word associated to a reduced expression of the longest element in the Weyl group of $C$, which has been proved in \cite{IM} for a special case and in \cite{Jos} for the general case. The Feigin homomorphism provides a realization of ${\mathbf U}^+$ as a subalgebra of a quantum torus, which leads to a construction of monomial bases for the quantum groups of finite type (cf. \cite{Rei}, \cite{Fu}).

In 1990, Ringel \cite{R90a} introduced the Hall algebra of a finite dimensional algebra $A$ over a finite field, and he \cite{Rin90,R90a} proved that if $A$ is a representation-finite hereditary algebra, the Ringel--Hall algebra of $A$ provides a realization of the positive part of the corresponding quantum group. For any hereditary algebra $A$, Green \cite{Gr95} proved that the corresponding quantum group ${\mathbf U}^+$ can be realized as a subalgebra, called composition subalgebra, of the Ringel--Hall algebra of $A$.
The Feigin homomorphism $\bF_{\mathbf{i}}$ has been generalised by Berenstein and Rupel \cite{Ber} to the Ringel--Hall algebra of a hereditary algebra. For a special choice of the word $\mathbf{i}$, the corresponding Feigin homomorphism $\bF_{\mathbf{i}}$ coincides with Reineke homomorphism (also called integration map) on Hall algebra (cf. \cite[Lemma 6.1]{Rei2}, \cite[Example
10]{Fei}).

\subsection{Feigin's map and cluster algebras}
Cluster algebras were introduced by Fomin and Zelevinsky~\cite{FZ} with the aim to set up a combinatorial framework for the study of total positivity  in algebraic groups and canonical bases in quantum groups. The quantum cluster algebras were introduced by Berenstein and Zelevinsky~\cite{BZ05}. The Caldero--Chapoton map (called cluster character) is introduced in \cite{CC} to establish the relation between categories of representations  of quivers and cluster algebras. The multiplication theorem of cluster characters given by Caldero and Keller~\cite{CK2005} indicates that there should be a close relation between cluster algebras and Hall algebras. The quantum cluster characters were formulated by Rupel \cite{Rupel1}, and the mutation multiplication formulas of quantum cluster characters for any acyclic valued quivers were obtained by Rupel \cite{Rupel2}.

Using the generalised Feigin homomorphism, Berenstein and Rupel \cite{Ber} introduced the generalised quantum cluster characters and obtained a quantum cluster structure on the quantum unipotent cell corresponding to the square of a Coxeter element. Recently, by using the integration homomorphism on the Hall algebra of the morphism category of projective modules over a hereditary algebra, as well as its bialgebra structure, Fu, Peng and Zhang~\cite{FPZ} provided an intrinsic realization of quantum cluster characters via Hall algebra approach. Subsequently, Chen, Ding and Zhang \cite{CDZ} developed the approach used in ~\cite{FPZ} and obtained the quantum version of Caldero--Keller's cluster multiplication theorem. In a word, the Feigin homomorphism plays a significant role in studying the relations between quantum groups, Hall algebras and quantum cluster algebras.

The topic of cluster realization of the entire quantum group has received many attentions and has been extensively studied. Schrader and Shapiro \cite{SS} provided a cluster algebra realisation of the quantum group ${\mathbf U}$ of type ${\rm A}$, that is, they constructed an explicit algebra embedding of ${\mathbf U}$ into a quotient of a quantum cluster algebra. Goncharov and Shen \cite{GS} have generalised the construction of Schrader and Shapiro to other finite types.

\subsection{Quantum symmetric pairs}

The notion of quantum symmetric pairs was introduced by Letzter \cite{Let99,Let02} (see \cite{Ko14} for an extension to Kac--Moody setting), which is a pair consisting of a quantum group ${\mathbf U}$ and its coideal subalgebra ${\mathbf U}^\imath$. The subalgebra ${\mathbf U}^\imath$ is called an $\imath$quantum group (and also called quantum symmetric pair coideal subalgebra in \cite{Let99,Ko14}). Following the classification of symmetric pairs of complex simple Lie algebras, the $\imath$quantum groups are classified by Satake diagrams. Quantum groups can be viewed as $\imath$quantum groups of diagonal types. So, the notion of $\imath$quantum groups is a vast generalisation of the notion of quantum groups. Many fundamental constructions and applications of quantum groups have found their generalisations in $\imath$quantum groups. See \cite{Wang} for a survey of some recent progress on quantum symmetric pairs
and applications. For example, Lu and Wang \cite{LW19a,LW20a} have succeeded in initiating an $\imath$Hall algebra approach for the categorification of quasi-split $\imath$quantum groups of Kac--Moody type. In fact, they use Hall algebras to realize the \emph{universal} quasi-split $\imath$quantum groups $\tUi$ which are introduced in \cite{LW19a} as subalgebras of the Drinfeld double quantum groups $\tU$; the usual $\imath$quantum groups $\Ui=\Ui_\bvs$ (depending on parameters $\bvs$) are reproduced by central reduction of $\tUi$.

The constructions of Schrader--Shapiro and Goncharov--Shen can be interpreted as the cluster realisations of $\imath$quantum groups of diagonal types. Thus, it is natural to expect cluster realisations of general $\imath$quantum groups such that fundamental constructions of $\imath$quantum groups admit
cluster interpretations. Recently, Song \cite{Song} has provided a cluster realisation of split universal $\imath$quantum group $\widetilde{\bU}^\imath$ of type ${\rm AI}$.
That is, he gave an algebra embedding of $\widetilde{\bU}^\imath$ of type ${\rm AI}$ into a quantum cluster algebra. The key ingredient of this algebra embedding is to get an algebra homomorphism from $\imath$quantum group $\widetilde{\bU}^\imath$ to a certain quantum torus.

\subsection{Goal}
In order to get cluster realisations of any split universal $\imath$quantum groups, as a first step, in this paper we define an algebra homomorphism from universal $\imath$quantum group of split type to a certain quantum torus associated to a valued quiver $Q$, which can be viewed as an $\imath$version of the Feigin homomorphism.

\subsection{Strategy}
We know universal $\imath$quantum group $\tUi$ of split type is defined as a subalgebra of Drinfeld double quantum group $\tU$, and  a Serre type presentation of $\tUi$ is established in \cite{CLW18,LW20a}, the most important relation is the $\imath$Serre relation \eqref{iserrerelation} formulated via $\imath$divided powers \cite{BW18a,BeW18,LW20a}.

The $\imath$analog of Feigin homomorphism $\varphi_{\mathbf{i}}$ for a word $\mathbf{i}$ in $\I$ is defined on generators of $\tUi$, and the main strategy is to prove that $\varphi_{\mathbf{i}}$ preserves $\imath$Serre relations, this is much more difficult than Feigin's map $\bF_{\mathbf{i}}$ for $\U^+$. For this purpose,
we first derive the closed formulas for the images of the $\imath$divided powers in the quantum torus; see Propositions \ref{prop:idivided-even}--\ref{prop:idivided-odd}. Next we convert the $\imath$Serre relation (which are formulated via $\imath$divided powers) in the quantum torus, and eventually reduce the proof further to some combinatorial identities (see \eqref{ou0}, \eqref{evenp=1} and \eqref{oddbdym}), for example: for any $m\geq0$, $0\leq |b|,\ell\leq m$,
\begin{align*}
     &\sum_{k=0}^{m} \sum\limits_{s=0}^{2k}
           [b(2\ell+1-4k+4s)]_{v_i}\qbinom{2m+1}{2k}_{v_i}\Big(\qbinom{2k-1}{s}_{v_i^2}+\qbinom{2k-1}{s-1}_{v_i^2}\Big)\qbinom{2m-2k+1}{m-\ell-s}_{v_i^2}=0,
    \end{align*}
    \begin{align*}
       \begin{split}
     \sum_{k=0}^{m} \sum\limits_{s=0}^{2k+1}
            [b(2\ell-1-4k+4s)]_{v_i}\qbinom{2m+1}{2k+1}_{v_i}\Big(\qbinom{2k}{s}_{v_i^2}+\qbinom{2k}{s-1}_{v_i^2}\Big)\qbinom{2m-2k}{m-\ell-s}_{v_i^2}=0.
             \end{split}
    \end{align*}

It turns out to be difficult to prove these quantum binomial identities directly; readers are encouraged to take on the challenge and attempt to prove them directly. We find a way around this by first proving them for $b=m$ and $0\leq \ell\leq m$, which is equivalent to establishing the Feigin homomorphism $\varphi_{\mathbf{i}}$ for the associated quiver $Q$ without $2$-cycles; see Proposition \ref{teshuqx}. Then we prove that this is enough to establish the Feigin homomorphism $\varphi_{\mathbf{i}}$ for general valued quiver $Q$ by using the non-standard Serre--Lusztig relations given in \cite{CLW21}. As a consequence, these general combinatorial identities hold; see Proposition \ref{cor:identities}.

The Feigin homomorphism $\varphi_{\mathbf{i}}$ for $\mathbf{i}=(1,2,\dots,N)$, where $\I=\{1,2,\dots,N\}$, directly induces the integration map $\varphi$ for $\Ui$ with some distinguished parameters $\bvs^\diamond$ given in \eqref{eq:disting-para}; see Corollary \ref{cor:iQG-TQ}.
It is remarkable that this homomorphism $\varphi$ has been obtained in \cite{Go24} for the $q$-Onsager algebra (the $\imath$quantum group $\Ui$ of affine $\mathfrak{sl}_2$).

    \subsection{Organization}
The paper is organized as follows: we recall the definitions and properties of quantum groups, $\imath$quantum groups, quantum tori, and present the main result in Section \ref{sec:main-result}. In Section \ref{sec:i-divided-QT}, we formulate the expansion formulas for $\imath$divided powers in quantum torus. Section \ref{sec:iSerre} is devoted to expanding the $\imath$Serre relations in quantum torus, and reducing the proof to some combinatorial identities. In Section \ref{sec:without2-loop}, we prove the main theorem for quivers without $2$-cycles.
We complete the proof of the main theorem in Section \ref{sec:proof-Main}. In Appendix \ref{sec:Comb-identities}, we provide several necessary combinatorial identities.

\vspace{0.5cm}
\noindent{\bf Acknowledgments.}
M. Lu is partially supported by the National Natural Science Foundation of China (No. 12171333, 12261131498).
S. Ruan is partially supported by the Natural Science Foundation of Xiamen (No. 3502Z20227184),
Fujian Provincial Natural Science Foundation of China (Nos. 2024J010006 and 2022J01034),
the National Natural Science Foundation of China (No. 12271448), and the Fundamental
Research Funds for Central Universities of China (No. 20720220043).
H. Zhang is partially supported by the National Natural Science Foundation of China (No. 12271257) and the Natural Science Foundation of Jiangsu Province of China (No. BK20240137).

\section{Main result}
\label{sec:main-result}

In this section, we recall the definitions of quantum groups, $\imath$quantum groups and quantum tori, and then formulate the main result.
\subsection{Universal $\imath$quantum groups of split type}
Let $\I=\{1,2,\dots,N\}$ be the index set.
We always assume that $C=(c_{ij})_{i,j \in \I}$ is a symmetrizable generalized Cartan matrix (GCM) with a symmetrizer $D={\rm diag}(d_i\mid i\in\I)$.
That is, $DC$ is symmetric.
Denote by $\fg$ the Kac--Moody Lie algebra corresponding to $C$.

Throughout the paper, let $\bv$ be an indeterminant and $\mathbb{Q}(v)$ be the rational function field. We write $[a, b]=ab-ba$ for two elements $a,b$ in an algebra and set $v_i=v^{d_i}$ for each $i\in\I$. For any integer $n$, define
$$[n] =\frac{v^n-v^{-n}}{v-v^{-1}}.$$
Thus $[-n]=-[n]$.
For any positive integer $m$,  define $[m]^! =\prod\limits_{i=1}^m [i]$ and write $[0]^!=1$.
For any integers $n,d$, define
\begin{align*}
\qbinom{n}{d}  &=
\begin{cases}
\frac{[n][n-1]\ldots [n-d+1]}{[d]^!}, & \text{ if }d > 0,
\\
1,& \text{ if }d=0,\\
0, & \text{ if }d<0.
\end{cases}
\end{align*}
For any element $c\in\mathbb{Q}(v)$, we write $[n]_{c}$, $[m]^!_c$ and $\qbinom{n}{d}_c$ for the evaluations of $[n]$, $[m]^!$ and $\qbinom{n}{d}$ at $v=c$, respectively.

For any integers $a,b,t$ with $t\geq0$, we have the following formulas (cf. \cite[\S0.2]{Jan}, \cite[\S1.3.1]{Lus93})
\begin{align}
\label{fkhgs}
\qbinom{a}{t}=v^{-\epsilon t}\qbinom{a-1}{t}+v^{\epsilon(a-t)}\qbinom{a-1}{t-1},
\\
\label{2.2}
\qbinom{a+b}{t}=\sum_{t_1+t_2=t}
v^{\epsilon(at_2-bt_1)}\qbinom{a}{t_1}\qbinom{b}{t_2},
\end{align}
where $\epsilon=\pm1$.

The \emph{Drinfeld double quantum group} $\tU := \tU_\bv(\fg)$ is defined to be the $\Q(\bv)$-algebra generated by $E_i,F_i, \tK_i,\tK_i'$, $i\in \I$, 
subject to the following relations {for $i,j\in\I$}:
\begin{align}
	[E_i,F_j]= \delta_{ij} \frac{\tK_i-\tK_i'}{\bv_i-\bv_i^{-1}},  &\qquad [\tK_i,\tK_j]=[\tK_i,\tK_j']  =[\tK_i',\tK_j']=0,
	\label{eq:KK}
	\\
	\tK_i E_j=\bv_i^{c_{ij}} E_j \tK_i, & \qquad \tK_i F_j=\bv_i^{-c_{ij}} F_j \tK_i,
	\label{eq:EK}
	\\
	\tK_i' E_j=\bv_i^{-c_{ij}} E_j \tK_i', & \qquad \tK_i' F_j=\bv_i^{c_{ij}} F_j \tK_i',
	\label{eq:K2}
\end{align}
and the quantum Serre relations {for $i\neq j \in \I$}:
\begin{align}
	& \sum_{r=0}^{1-c_{ij}} (-1)^r  E_i^{(r)} E_j  E_i^{(1-c_{ij}-r)}=0,
	& \sum_{r=0}^{1-c_{ij}} (-1)^r   F_i^{(r)} F_j  F_i^{(1-c_{ij}-r)}=0,
	\label{eq:serre2}
\end{align}
where \[
F_i^{(n)} =F_i^n/[n]^!_{v_i}, \quad E_i^{(n)} =E_i^n/[n]^!_{v_i}, \quad \text{ for } n\ge 1, i\in \I.
\]
Note that $\tK_i \tK_i'$ is central in $\tU$ for any $i$.

Analogously, the Drinfeld-Jimbo \emph{quantum group} $\bU:=\bU_\bv(\fg)$ is defined to be the $\Q(v)$-algebra generated by $E_i,F_i, K_i, K_i^{-1}$, $i\in \I$, subject to the relations modified from \eqref{eq:KK}--\eqref{eq:serre2} with $\tK_i$ and $\tK_i'$ replaced by $K_i$ and $K_i^{-1}$, respectively.

The {\em universal $\imath$quantum group (of split type)} $\widetilde{\bU}^\imath$ is defined in \cite{LW19a} (cf.  \cite{Let99,Ko14})
to be the $\Q(v)$-subalgebra of $\tU$ generated by
\[
B_i= F_i +  E_{i} \tK_i'~~\text{and}~~
\tk_i = \tK_i \tK_{i}'~~\text{for~any}~i \in \I.
\]
Let $\Z/2\Z=\{\ov{0},\ov{1}\}$. For each $i\in \I$, according to \cite{BW18a, BeW18,CLW18,LW20a}, the {\em $\imath${}divided powers} of $B_i$ are defined to be
\begin{eqnarray}
	&&\ff_{i,\odd}^{(m)}=\frac{1}{[m]^!_{v_i}}\left\{ \begin{array}{ccccc} B_i\prod\limits_{s=1}^k (B_i^2-v_i\tk_i[2s-1]_{v_i}^2 ) &  \qquad\text{if }m=2k+1, \\
		\prod\limits_{s=1}^k (B_i^2-v_i\tk_i[2s-1]_{v_i}^2) &\text{if }m=2k; \end{array}\right.
	\label{eq:iDPodd}
	\\
	&&\ff_{i,\ev}^{(m)}= \frac{1}{[m]^!_{v_i}}\left\{ \begin{array}{ccccc} B_i\prod\limits_{s=1}^k (B_i^2-v_i\tk_i[2s]^2_{v_i} ) & \qquad\text{if }m=2k+1,\\
		\prod\limits_{s=1}^{k} (B_i^2-v_i\tk_i[2s-2]_{v_i}^2) &\text{if }m=2k. \end{array}\right.
	\label{eq:iDPev}
\end{eqnarray}

\begin{theorem}[\text{\cite[Theorem 4.2]{LW20a}}; see also \text{\cite[Theorem~3.1]{CLW18}}]
	\label{thm:Serre}
	Fix $\ov{p}_i\in \Z/2\Z$ for each $i\in \I$. The {universal $\imath$quantum group} $\widetilde{\bU}^\imath$ has a presentation with generators $B_i, \tk_i$ $(i\in \I)$ and the relations for any $i,j\in\I$:
	\begin{align}
&\tk_i\tk_j=\tk_j\tk_i,\quad\quad\tk_i B_j=B_j\tk_i,\\
			&\sum_{n=0}^{1-c_{ij}} (-1)^n  B_{i, \overline{p_i}}^{(n)}B_j B_{i,\overline{c_{ij}}+\overline{p}_i}^{(1-c_{ij}-n)}=0,\quad\text{if}~~ i\neq j.
		\label{iserrerelation}
	\end{align}
\end{theorem}
The relations \eqref{iserrerelation} are called the $\imath$Serre relations. The Serre--Lusztig relations (also called higher order $\imath$Serre relations) for $\imath$quantum groups are given in \cite{CLW21}, where they also give the following non-standard Serre--Lusztig relations.

\begin{proposition}[\text{\cite[Theorem B]{CLW21}}]
  \label{prop:SLinduct}
Fix $\ov{p}_i\in \Z/2\Z$ for each $i\in \I$. For any $i \neq j\in \I$ and non-negative integer $t$, the following identity holds in $\tUi$
\begin{align}
\sum_{s+s'=1-c_{ij}+2t} (-1)^s B^{(s)}_{i, \overline{p_i}} B_j B_{i,\overline{c_{ij}}+\overline{p}_i}^{(s')}=0.
\label{eq:Serre2t}
\end{align}
\end{proposition}

Let $\bvs=(\vs_i)\in  (\Q(\bv)^\times)^{\I}$.
	By \cite{Let99,Ko14}, the Letzter's $\imath$quantum group  $\Ui:=\Ui_{\bvs}$ is the $\Q(v)$-subalgebra of $\bU$ generated by
	\[
	B_i= F_i+\vs_i E_{i}K_i^{-1},\quad\forall i\in\I.
	\]
    By \cite[Proposition 6.2]{LW19a}, the $\Q(v)$-algebra $\Ui$ is isomorphic to the quotient of $\tUi$ by the ideal generated by $\tk_i-\vs_i$ ($i\in\I$). It is well known that the $\Q(v)$-algebra $\Ui_\bvs$ (up to some field extension) is isomorphic
to $\tUi_{\bvs^\diamond}$ for some distinguished parameters $\bvs^\diamond$ (cf. \cite{Let99}, \cite[Proposition 9.2]{Ko14}; see also \cite[Proposition 5]{CLW18}).
  In this paper, we always assume the distinguished parameters $\bvs^\diamond$ to be \begin{align}
  \label{eq:disting-para}
     \vs_i^\diamond=-v_i^{-1}(v_i-v_i^{-1})^2,\qquad \forall i\in\I,
  \end{align}
  and $\Ui=\Ui_{\bvs^\diamond}$; compare with \cite[eq. (7.1)]{LW21}.

It is remarkable that $\tU$ and $\tUi$ given here are different from those defined in \cite{LW19a}, since we do not require the elements  $\tK_i$, $\tK_i'$, $\tk_i$ ($i\in\I$) to be invertible (cf. \cite{Qin16}).

\subsection{Quantum torus}
Recall the GCM $C=(c_{ij})_{i,j \in \I}$ and its symmetrizer $D=\diag(d_i\mid i\in\I)$. According to \cite[\S0.1]{Deng}, we associate with $C$ a valued graph $\Gamma=\Gamma(C)$. Explicitly, the vertex set of $\Gamma$ is just $\I$, and for any vertices $i\neq j$ there are $n_{ij}=-\frac{d_ic_{ij}}{d_{i,j}}$ parallel edges connecting $i$ and $j$, where $d_{i,j}$ is the least common multiple of $d_i$ and $d_j$. Moreover, each vertex $i$ is assigned the value $d_i$, and each edge $\rho$ between any vertices $i,j$ is given the value $m_\rho=d_{i,j}$.

Given an orientation arrangement of the valued graph $\Gamma$, we have a valued quiver $Q$ with the vertex set $Q_0=\I$ and arrow set $Q_1$. By the definition of $\Gamma$, $Q$ has no loops. For an arrow $\rho\in Q_1$, denote the tail and head of $\rho$ by ${\rm t}\rho$ and ${\rm h}\rho$, respectively. Let $\mathbb{Z}\I$ be the free abelian group with basis $\{e_i~|~i\in\I\}$. The {\em Euler form} of $Q$ is defined to be the bilinear form $\lr{-,-}_Q:\mathbb{Z}\I\times \mathbb{Z}\I\rightarrow\mathbb{Z}$ given by
$$\lr{{\bf x},{\bf y}}_Q=\sum\limits_{i\in Q_0}d_ix_iy_i-\sum\limits_{\rho\in Q_1}m_{\rho}x_{{\rm t}\rho}y_{{\rm h}\rho}$$
for any ${\bf x}=\sum\limits_{i\in\I}x_ie_i$ and ${\bf y}=\sum\limits_{i\in\I}y_ie_i$. It is easy to see that $\lr{e_i,e_j}_Q=-r_{ij}d_{i,j}$ for any $i\neq j$, where $r_{ij}$ is the number of arrows from $i$ to $j$ in $Q$. Thus, $d_i~|~\lr{e_i,e_j}_Q$ and $d_j~|~\lr{e_i,e_j}_Q$.

The {\em symmetric Euler form} $(-,-)_Q:\mathbb{Z}\I\times \mathbb{Z}\I\rightarrow\mathbb{Z}$ is defined by
$$({\bf x},{\bf y})_Q:=\lr{{\bf x},{\bf y}}_Q+\lr{{\bf y},{\bf x}}_Q$$ for any ${\bf x},{\bf y}\in\mathbb{Z}\I$. It is easy to see that
$$({\bf x},{\bf y})_Q={\bf x}^{\rm tr}(DC){\bf y}$$
for any ${\bf x},{\bf y}\in\mathbb{Z}\I$. In particular, we have $(e_i,e_j)_Q=d_ic_{ij}$ for any $i,j\in\I$.

\begin{definition}
{Let $\mathbf{i}=(i_1i_2\cdots i_r)$ be a word made up of elements in $\I$. The {quantum torus} $\mathcal{T}_\mathbf{i}(Q)$ is defined to be the $\Q(v)$-algebra generated by $x_1^{\pm1},\ldots,x_r^{\pm1}$ subject to the relations
\begin{align*}
&x_ix_i^{-1}=x_i^{-1}x_i=1~~\text{for~any}~1\leq i\leq r,\\
&x_lx_k=v^{\lr{e_{i_k},e_{i_l}}_Q-\lr{e_{i_l},e_{i_k}}_Q}x_kx_l~~\text{for~any}~1\leq k,l\leq r.
\end{align*}}
\end{definition}

By definition, clearly, there exists an algebra involution
\begin{align}\label{involution}
   \sigma: \mathcal{T}_\mathbf{i}(Q)\longrightarrow \mathcal{T}_\mathbf{i}(Q),\quad x_i\longmapsto x_i^{-1}, \quad \text{for~any\;} 1\leq i\leq r.
\end{align}

\subsection{Main result}
Now we can give our main result of this paper as follows.
{\begin{theorem}\label{mainthm}
Given a valued quiver $Q$ associated to $C$ and a word $\mathbf{i}=(i_1i_2\cdots i_r)$,
there exists an algebra homomorphism
	\begin{align*}
		\varphi_{\mathbf{i}}: \tUi &\longrightarrow \mathcal{T}_\mathbf{i}(Q)
	\end{align*}
defined on generators by
	\begin{align}
    \label{eq:phi-form1}
    &B_i \longmapsto \sum\limits_{1\leq k\leq r:i_k=i}(x_k+x_k^{-1}),
    \\
    \label{eq:phi-form2}
    &\tk_i\longmapsto-v_i^{-1}(v_i-v_i^{-1})^2(\sum\limits_{1\leq k\leq r:i_k=i}x_k)(\sum\limits_{1\leq k\leq r:i_k=i}x_k^{-1}),
    \end{align}
for any $i\in\I$.
\end{theorem}}

\begin{proof}
Note that $(\sum\limits_{1\leq k\leq r:i_k=i}x_k)(\sum\limits_{1\leq k\leq r:i_k=i}x_k^{-1})$ is central in $\mathcal{T}_\mathbf{i}(Q)$ for any $i\in\I$.
So we only need to prove the $\imath$Serre relations \eqref{iserrerelation} are preserved under $\varphi_{\mathbf{i}}$, this will be  done in the remainder of this paper. 
\end{proof}

Recall $\Ui=\Ui_{\bvs^\diamond}$. Let $\ct(Q)$ be the quantum torus associated to the word $\mathbf{i}=(1,2,\dots,N)$ with $\I=\{1,2,\dots,N\}$.
Then we have the following corollary.
\begin{corollary}
\label{cor:iQG-TQ}
Given a valued quiver $Q$ associated to the GCM $C$, there exists an algebra homomorphism
\begin{align*}
\varphi:\Ui\longrightarrow \ct(Q),~~ B_i\longmapsto x_i+x_i^{-1}, \quad \forall i\in\I.
\end{align*}
\end{corollary}
Following \cite{Rei2}, we also call $\varphi:\Ui\rightarrow \ct(Q)$ the integration map of $\Ui$.

\begin{remark}
    The $\imath$quantum group $\Ui$ of affine $\mathfrak{sl}_2$ is also called to be $q$-Onsager algebra (see \cite{T93}). For the Kronecker quiver $Q:\xymatrix{1\ar@<0.5ex>[r] \ar@<-0.5ex>[r]& 2}$, the morphism $\varphi:\Ui\rightarrow \ct(Q)$ in Corollary \ref{cor:iQG-TQ}
    for this special case was obtained in \cite{Go24}.
\end{remark}

\section{$\imath$Divided powers in the quantum torus}
\label{sec:i-divided-QT}

Let $\mathbf{i}=(i_1i_2\cdots i_r)$ be a word made up of elements in $\I$, and $\ct_\mathbf{i}(Q)$ the quantum torus. In this section, we shall define $\imath$divided powers in $\ct_{\mathbf{i}}(Q)$ and formulate their expansion formulas.

\subsection{$\imath$Divided powers in quantum torus}
For each $i\in\I$, set in $\ct_\mathbf{i}(Q)$:  $$X_i=\sum\limits_{1\leq k\leq r:i_k=i} x_k,\quad Y_i=\sum\limits_{1\leq k\leq r:i_k=i} x_k^{-1},  {\text{\;\;and\;\;} } u_i=X_iY_i=Y_iX_i.$$
It is easy to see that for any $i,j\in\I$ and any $t\geq0$, we have
\begin{align}
X_jX_i^t&=v^{t(\lr{e_i,e_j}_Q-\lr{e_j,e_i}_Q)}X_i^tX_j,~~&Y_jY_i^t=v^{t(\lr{e_i,e_j}_Q-\lr{e_j,e_i}_Q)}Y_i^tY_j,\;\; \\
X_jY_i^t&=v^{-t(\lr{e_i,e_j}_Q-\lr{e_j,e_i}_Q)}Y_i^tX_j,~~&Y_jX_i^t=v^{-t(\lr{e_i,e_j}_Q-\lr{e_j,e_i}_Q)}X_i^tY_j,\\
u_iX_j&=X_ju_i,~~&u_iY_j=Y_ju_i.\qquad\qquad\qquad\qquad
\end{align}
For any positive integer $m$, set $$y_{i,m}=X_i^m+Y_i^{m}~~\text{and}~~y_{i,-m}=0.$$ For convenience, we set $y_{i,0}=1$.  Then we have
\begin{align}
y_{i,m} y_{i,n}=y_{i,n} y_{i,m}=y_{i,m+n}+y_{i,m-n}u_i^n+\delta_{m,n}u_i^n
\end{align}
for any integers $m\geq n>0$.
We simply write $y_i$ for $y_{i,1}$, then
\begin{align}
y_i=X_i+Y_i,\qquad y_i^2=y_{i,2}+2u_i.
\end{align}

For any $i,j\in\I$ with $j\neq i$, set $\n_{ij}=\lr{e_i,e_j}_Q$ and $\n_{ji}=\lr{e_j,e_i}_Q$, then $\n_{ij}+\n_{ji}=d_ic_{ij}$. Let
$\delta'_{x,0}=1-\delta_{x,0}$, i.e.,
\begin{align}
\delta'_{x,0}=\begin{cases} 1\;\;&\text{if }{~x\neq0};\\
                     0 &\text{otherwise.}\end{cases}
\end{align}
\begin{lemma}\label{yxy}
For any $i\neq j\in\I$, $r,t\in\N$, we have
\begin{align*}
y_{i,r}\cdot X_j\cdot y_{i,t}=&v^{t(\n_{ij}-\n_{ji})}X_i^{r+t}X_j+v^{-t(\n_{ij}-\n_{ji})}Y_i^{r+t}X_j\\
&+\delta'_{rt,0}\big(v^{t(\n_{ij}-\n_{ji})}Y_i^rX_i^tX_j+v^{-t(\n_{ij}-\n_{ji})}X_i^rY_i^tX_j\big).
\end{align*}
\end{lemma}

\begin{proof}
We only prove the case of $r,t>0$, and the proof for $r=0$ or $t=0$ is similar.
By definition,
\begin{align*}
y_{i,r}\cdot X_j\cdot y_{i,t}&=(X_i^r+Y_i^r)X_j(X_i^t+Y_i^t)\\
&=(X_i^r+Y_i^r)(v^{t(\n_{ij}-\n_{ji})}X_i^tX_j+v^{-t(\n_{ij}-\n_{ji})}Y_i^tX_j)\\
&=v^{t(\n_{ij}-\n_{ji})}X_i^{r+t}X_j+v^{-t(\n_{ij}-\n_{ji})}Y_i^{r+t}X_j\\
&\quad+v^{t(\n_{ij}-\n_{ji})}Y_i^rX_i^tX_j+v^{-t(\n_{ij}-\n_{ji})}X_i^rY_i^tX_j.
\end{align*}
\end{proof}

For any $m\in\N$, inspired by \eqref{eq:iDPodd}--\eqref{eq:iDPev}, we
define the {\em $\imath$divided powers} of $y_i$ to be:
\begin{eqnarray*}
	&&\ffy_{i,\ev}^{(m)}= \frac{1}{[m]^!_{v_i}}\left\{ \begin{array}{ccccc} y_i\prod\limits_{s=1}^k (y_i^2+(v_i-v_i^{-1})^2[2s]_{v_i}^2 u_i) & \qquad\text{if }m=2k+1,\\
		\prod\limits_{s=1}^{k} (y_i^2+(v_i-v_i^{-1})^2[2s-2]_{v_i}^2 u_i) &\text{if }m=2k; \end{array}\right.
    \\\\
    &&\ffy_{i,\odd}^{(m)}=\frac{1}{[m]^!_{v_i}}\left\{ \begin{array}{ccccc} y_i\prod\limits_{s=1}^k (y_i^2+(v_i-v_i^{-1})^2[2s-1]_{v_i}^2 u_i) &  \qquad\text{if }m=2k+1, \\
		\prod\limits_{s=1}^k (y_i^2+(v_i-v_i^{-1})^2[2s-1]_{v_i}^2u_i) &\text{if }m=2k. \end{array}\right.
\end{eqnarray*}
Then the $\imath$divided powers in $\ct_{\mathbf{i}}(Q)$ satisfy the following recursive relations:
\begin{align}
y_i \ffy_{i,\ev}^{(2k-1)}&=[2k]_{v_i}\ffy_{i,\ev}^{(2k)}\\
\label{ji2kj1}
y_i \ffy_{i,\ev}^{(2k)}&=[2k+1]_{v_i}\ffy_{i,\ev}^{(2k+1)}-(v_i-v_i^{-1})^2[2k]_{v_i}\ffy_{i,\ev}^{(2k-1)}u_i;
\end{align}
and
\begin{align}
y_i \ffy_{i,\odd}^{(2k)}&=[2k+1]_{v_i}\ffy_{i,\odd}^{(2k+1)}\\
y_i \ffy_{i,\odd}^{(2k+1)}&=[2k+2]_{v_i}\ffy_{i,\odd}^{(2k+2)}-(v_i-v_i^{-1})^2[2k+1]_{v_i}\ffy_{i,\odd}^{(2k)}u_i.
\end{align}

The motivation of the  definition of $\imath$divided powers in the quantum torus comes from
\begin{equation}
    \varphi_\mathbf{i}(B_{i,\ov{0}}^{(m)})= y_{i,\ov{0}}^{(m)}~\text{and}~\varphi_\mathbf{i}(B_{i,\ov{1}}^{(m)})= y_{i,\ov{1}}^{(m)}
\end{equation}
for any $i\in\I$.
So in order to  prove that the $\imath$Serre relations \eqref{iserrerelation} are preserved by $\varphi_{\mathbf{i}}$, it is enough to prove
\begin{align}\label{yzgx0}
\sum_{n=0}^{1-c_{ij}} (-1)^n  y_{i,\overline{p_i}}^{(n)}\cdot (X_j+Y_j)\cdot y_{i,\overline{c_{ij}}+\overline{p}_i}^{(1-c_{ij}-n)} &=0
\end{align}
for any $i,j\in\I$ with $j\neq i$ and $\ov{p}_i\in \Z/2\Z$.

\subsection{The expansion formulas}

In order to provide the expansion formulas of $\imath$divided powers $y_{i,\ov{0}}^{(m)}$ and $y_{i,\ov{1}}^{(m)}$, we need the following combinatorial identity.

\begin{lemma}\label{Lus} For any integers $a,n$ with $n\geq0$, we have
$$\qbinom{a+1}{n}=\qbinom{a-1}{n}+(v^{a}+v^{-a})\qbinom{a-1}{n-1}+\qbinom{a-1}{n-2}.$$
\end{lemma}
\begin{proof} If $n=0$, it is clear. If $n>0$, by (\ref{fkhgs}), we obtain
      \begin{align*}\qbinom{a+1}{n}&=v^{n}\qbinom{a}{n}+v^{n-(a+1)}\qbinom{a}{n-1}\\
      &=v^{n}\big(v^{-n}\qbinom{a-1}{n}+v^{a-n}\qbinom{a-1}{n-1}\big)+v^{n-(a+1)}\big(v^{-(n-1)}\qbinom{a-1}{n-1}+v^{a-(n-1)}\qbinom{a-1}{n-2}\big)\\
      &=\qbinom{a-1}{n}+(v^{a}+v^{-a})\qbinom{a-1}{n-1}+\qbinom{a-1}{n-2}.
   \end{align*}
\end{proof}

\begin{proposition}
\label{prop:idivided-even}
For any integer $m\geq0$,
\begin{eqnarray*}
	&&\ffy_{i,\ev}^{(m)}= \frac{1}{[m]^!_{v_i}}\begin{cases}
    \sum\limits_{s=0}^k\qbinom{m}{s}_{v_i^2}y_{i,m-2s}u_i^s
    & \text{if }m=2k+1,\\
		   \sum\limits_{s=0}^k(\qbinom{m-1}{s}_{v_i^2}+\qbinom{m-1}{s-1}_{v_i^2})y_{i,m-2s}u_i^s
        &\text{if }m=2k. \end{cases}
\end{eqnarray*}
\end{proposition}

\begin{proof}
First of all, by definition, we have $\ffy_{i,\ev}^{(0)}=1$, $\ffy_{i,\ev}^{(1)}=y_i$ and $\ffy_{i,\ev}^{(2)}=\frac{1}{[2]_{v_i}}(y_{i,2}+2u_i)$. So the results hold for $m=0,1,2$.

Now we assume that the results hold for any $m\leq 2k$.
We will prove that the results hold for $m=2k+1$ and $m=2k+2$ by induction.

A direct computation shows that
\begin{align*}
y_i \ffy_{i,\ev}^{(2k)}&=\frac{1}{[2k]^!_{v_i}}\sum\limits_{s=0}^{k}(\qbinom{2k-1}{s}_{v_i^2}+\qbinom{2k-1}{s-1}_{v_i^2})y_i y_{i,2k-2s}u_i^s\\
&=\frac{1}{[2k]^!_{v_i}}\sum\limits_{s=0}^{k}(\qbinom{2k-1}{s}_{v_i^2}+\qbinom{2k-1}{s-1}_{v_i^2})(y_{i,2k+1-2s}+y_{i,2k-1-2s}u_i)u_i^s\\
&=\frac{1}{[2k]^!_{v_i}}\Big(\sum\limits_{s=0}^{k}(\qbinom{2k-1}{s}_{v_i^2}+\qbinom{2k-1}{s-1}_{v_i^2}) y_{i,2k+1-2s}u_i^s\\
&\qquad\qquad+\sum\limits_{s=0}^{k-1}(\qbinom{2k-1}{s}_{v_i^2}+\qbinom{2k-1}{s-1}_{v_i^2}) y_{i,2k-1-2s}u_i^{s+1}\Big)\\
&=\frac{1}{[2k]^!_{v_i}}\Big(\sum\limits_{s=0}^{k}(\qbinom{2k-1}{s}_{v_i^2}+\qbinom{2k-1}{s-1}_{v_i^2}) y_{i,2k+1-2s}u_i^s\\
&\qquad\qquad+\sum\limits_{s=1}^{k}(\qbinom{2k-1}{s-1}_{v_i^2}+\qbinom{2k-1}{s-2}_{v_i^2})y_{i,2k+1-2s}u_i^{s}\Big),
\end{align*}
and
\begin{align*}
(v_i-v_i^{-1})^2[2k]_{v_i}\ffy_{i,\ev}^{(2k-1)}u_i
=&(v_i-v_i^{-1})^2[2k]_{v_i}\cdot\frac{1}{[2k-1]^!_{v_i}}\sum\limits_{s=0}^{k-1}\qbinom{2k-1}{s}_{v_i^2}y_{i,2k-1-2s}u_i^{s+1} \\
=&\frac{1}{[2k]^!_{v_i}}(v_i^{2k}-v_i^{-2k})^2 \sum\limits_{s=1}^{k}\qbinom{2k-1}{s-1}_{v_i^2}y_{i,2k+1-2s}u_i^{s}.
\end{align*}
Using \eqref{ji2kj1}, we obtain
\begin{align*}
\ffy_{i,\ev}^{(2k+1)}&= \frac{1}{[2k+1]_{v_i}}(y_i \ffy_{i,\ev}^{(2k)}+(v_i-v_i^{-1})^2[2k]_{v_i}\ffy_{i,\ev}^{(2k-1)}u_i)
\\
&=\frac{1}{[2k+1]^!_{v_i}}\sum\limits_{{{s=0}}}^k (\qbinom{2k-1}{s}_{v_i^2}+\qbinom{2k-1}{s-2}_{v_i^2}+(v_i^{4k}+v_i^{-4k})\qbinom{2k-1}{s-1}_{v_i^2})y_{i,2k+1-2s}u_i^s\\
&=\frac{1}{[2k+1]^!_{v_i}}\sum\limits_{s=0}^k \qbinom{2k+1}{s}_{v_i^2}y_{i,2k+1-2s}u_i^s,
\end{align*}
where the last equality follows from the following identity $$\qbinom{2k-1}{s}_{v_i^2}+\qbinom{2k-1}{s-2}_{v_i^2}+(v_i^{4k}+v_i^{-4k})\qbinom{2k-1}{s-1}_{v_i^2}=\qbinom{2k+1}{s}_{v_i^2},$$
which is due to Lemma \ref{Lus}.
Hence, the result holds for $m=2k+1$.

For $m=2k+2$, we have
\begin{align*}
\ffy_{i,\ev}^{(2k+2)}&= \frac{1}{[2k+2]_{v_i}}y_i \ffy_{i,\ev}^{(2k+1)}
\\
&=\frac{1}{[2k+2]^!_{v_i}}\sum\limits_{s=0}^k\qbinom{2k+1}{s}_{v_i^2}y_i y_{i,2k+1-2s}u_i^s\\
&=\frac{1}{[2k+2]^!_{v_i}}\sum\limits_{s=0}^k\qbinom{2k+1}{s}_{v_i^2}(y_{i,2k+2-2s}u_i^s+y_{i,2k-2s}u_i^{s+1}+\delta_{s,k}u_i^{s+1}).
\end{align*}
    Note that
    \begin{align*}
&\sum\limits_{s=0}^k\qbinom{2k+1}{s}_{v_i^2}y_{i,2k+2-2s}
u_i^{s}=y_{i,2k+2}+\sum\limits_{s=1}^k\qbinom{2k+1}{s}_{v_i^2}y_{i,2k+2-2s}u_i^{s},
\end{align*}
and
    \begin{align*}
&\sum\limits_{s=0}^k\qbinom{2k+1}{s}_{v_i^2}(y_{i,2k-2s}u_i^{s+1}+\delta_{s,k}u_i^{s+1})
\\
&=\sum\limits_{s=0}^{k-1}\qbinom{2k+1}{s}_{v_i^2} y_{i,2k-2s}u_i^{s+1}+\qbinom{2k+1}{k}_{v_i^2}\cdot 2u_i^{k+1}\\
&=\sum\limits_{s=1}^k\qbinom{2k+1}{s-1}_{v_i^2}y_{i,2k-2s+2}u_i^{s}+(\qbinom{2k+1}{k}_{v_i^2}+\qbinom{2k+1}{k+1}_{v_i^2})u_i^{k+1}.
\end{align*}

Hence,
\begin{align*}
\ffy_{i,\ev}^{(2k+2)}&=\frac{1}{[2k+2]^!_{v_i}}\Big(y_{i,2k+2}
+\sum\limits_{s=1}^k(\qbinom{2k+1}{s}_{v_i^2}+\qbinom{2k+1}{s-1}_{v_i^2})y_{i,2k+2-2s}u_i^{s}+\qbinom{2k+1}{k+1}_{v_i^2}\cdot 2u_i^{k+1}\Big)\\
&=\frac{1}{[2k+2]^!_{v_i}}\sum\limits_{s=0}^{k+1}(\qbinom{2k+1}{s}_{v_i^2}+\qbinom{2k+1}{s-1}_{v_i^2})y_{i,2k+2-2s}u_i^{s}.
\end{align*}
So the result holds for $m=2k+2$.
Therefore, we finish the proof.
\end{proof}

By similar arguments as above, we obtain the following.
\begin{proposition}
\label{prop:idivided-odd}
For any  integer $m\geq0$,
\begin{eqnarray*}
	&&\ffy_{i,\odd}^{(m)}= \frac{1}{[m]^!_{v_i}} \begin{cases}
    \sum\limits_{s=0}^k(\qbinom{m-1}{s}_{v_i^2}+\qbinom{m-1}{s-1}_{v_i^2})y_{i,m-2s} u_i^s
    & \text{if }m=2k+1,\\
		   \sum\limits_{s=0}^k\qbinom{m}{s}_{v_i^2}y_{i,m-2s}u_i^s
        &\text{if }m=2k. \end{cases}
\end{eqnarray*}
\end{proposition}

Recall from \eqref{involution} that the involution $\sigma$ on $\mathcal{T}_\mathbf{i}(Q)$ satisfies $\sigma(X_i)=Y_i$ for any $i\in\I$. Hence $\sigma(u_i)=u_i$ and $\sigma(y_{i,m})=y_{i,m}$ for any $i\in\I, m\in\N$, and then
the $\imath$divided powers $\ffy_{i,\ev}^{(m)}$ and $\ffy_{i,\odd}^{(m)}$ are preserved by $\sigma$.
Therefore, the equation \eqref{yzgx0} can be simplified, and we have the following proposition.

\begin{proposition}\label{prop3.5}
The $\imath$Serre relations \eqref{iserrerelation} holds if and only if
 \begin{align}\label{yzgx}
\sum_{n=0}^{1-c_{ij}} (-1)^n  y_{i,\overline{p_i}}^{(n)}\cdot X_j\cdot y_{i,\overline{c_{ij}}+\overline{p}_i}^{(1-c_{ij}-n)} &=0
\end{align}
for any $i,j\in\I$ with $j\neq i$ and $\ov{p}_i\in \Z/2\Z$.
\end{proposition}


\section{$\imath$Serre relations in the quantum torus}
\label{sec:iSerre}

In this section, we shall expand
\eqref{yzgx} using the formulas on $\imath$divided powers given in Propositions \ref{prop:idivided-even}--\ref{prop:idivided-odd}. After reorganizations of the computations, we reduce the proof of \eqref{yzgx} to the verification of certain quantum binomial identities.

\subsection{Even $c_{ij}$}
\label{subsec:even-even}
In this subsection, let us consider the proof of (\ref{yzgx}) for the case that $c_{ij}$ is even. Assume $c_{ij}=-2m$, $\n_{ij}=-ad_i$ and $\n_{ji}=-(a+2b)d_i$ for some non-negative integers $a,m$ and $-m\leq b\leq m$. Then by $\n_{ij}+\n_{ji}=d_ic_{ij}$, we have $m=a+b$. In particular, if $b=\pm m$, we can see that either $\n_{ij}=0$ or $\n_{ji}=0$.

\subsubsection{Even $c_{ij}$ and $\ov{p}_i=0$}

Let us rewrite  \eqref{yzgx} in this case:
    \begin{align}\label{Serre relation for Xj}
			\sum_{n=0}^{2m+1} (-1)^n  y_{i,\ov{0}}^{(n)}\cdot X_j\cdot y_{i,\ov{0}}^{(2m+1-n)} =0.
	\end{align}

    Note that \eqref{Serre relation for Xj}  is equivalent to
     \begin{align*}
			\sum_{k=0}^{m} y_{i,\ov{0}}^{(2k)}\cdot X_j \cdot y_{i,\ov{0}}^{(2m+1-2k)} =\sum_{k=0}^{m} y_{i,\ov{0}}^{(2k+1)}\cdot X_j \cdot y_{i,\ov{0}}^{(2m-2k)}.
	\end{align*}
    That is,
     \begin{align*}
			&\sum_{k=0}^{m}
            \frac{1}{[2k]^!_{v_i}}\sum\limits_{s=0}^k(\qbinom{2k-1}{s}_{v_i^2}+\qbinom{2k-1}{s-1}_{v_i^2})y_{i,2k-2s}u_i^s
            \cdot X_j \\ &\qquad\qquad\cdot
            \frac{1}{[2m-2k+1]^!_{v_i}}\sum\limits_{t=0}^{m-k}\qbinom{2m-2k+1}{t}_{v_i^2}y_{i,2m-2k+1-2t}u_i^t\\
            &=\sum_{k=0}^{m}  \frac{1}{[2k+1]^!_{v_i}}
    \sum\limits_{t=0}^k\qbinom{2k+1}{t}_{v_i^2}y_{i,2k+1-2t} u_i^t
    \cdot X_j \\ &\qquad\qquad \cdot
     \frac{1}{[2m-2k]^!_{v_i}}\sum\limits_{s=0}^{m-k}(\qbinom{2m-2k-1}{s}_{v_i^2}+\qbinom{2m-2k-1}{s-1}_{v_i^2})y_{i,2m-2k-2s}u_i^s.
	\end{align*}
Via replacing $k$ by $m-k$ on the right-hand side, it suffices to show
\begin{align}\label{first left and right}
			&\sum_{k=0}^{m} \sum\limits_{s=0}^k\sum\limits_{t=0}^{m-k}
            \phi(k,s,t) y_{i,2k-2s}
            \cdot X_j\cdot y_{i,2m-2k+1-2t} u_i^{s+t}\\\notag
            &=\sum_{k=0}^{m}
    \sum\limits_{s=0}^k \sum\limits_{t=0}^{m-k}
    \phi(k,s,t)
    y_{i,2m-2k+1-2t}
    \cdot X_j\cdot  y_{i,2k-2s} u_i^{s+t},
	\end{align}
where
    $$\phi(k,s,t):=\frac{1}{[2k]^!_{v_i}}\frac{1}{[2m-2k+1]^!_{v_i}}(\qbinom{2k-1}{s}_{v_i^2}+\qbinom{2k-1}{s-1}_{v_i^2})
            \qbinom{2m-2k+1}{t}_{v_i^2}.$$
 
    By Lemma \ref{yxy} and $\n_{ij}-\n_{ji}=2bd_i$, we have
      \begin{align*}
 y_{i,r}\cdot X_j\cdot y_{i,t}&=(v_i^{2bt}X_i^{r+t}+v_i^{-2bt}Y_i^{r+t}+\delta'_{rt,0}(v_i^{2bt}Y_i^rX_i^{t}+v_i^{-2bt}X_i^{r}Y_i^t))X_j.
 \end{align*}
Note that $u_i=X_iY_i$ and $u_iX_j=X_ju_i$. It suffices to show that the coefficients of $X_i^{2\ell+1}\cdot X_j\cdot u_i^{m-\ell}=X_i^{m+\ell+1}\cdot Y_i^{m-\ell}\cdot X_j$ on both sides of \eqref{first left and right} are identical for any $-m\leq \ell\leq m$. That is,
    \begin{align*}
      &\sum_{k=0}^{m} \sum\limits_{s=0}^k\sum\limits_{\begin{subarray}{c}
           t=0  \\
           t+s=m-\ell
      \end{subarray}}^{m-k}
           v_i^{2b(2m-2k+1-2t)} \phi(k,s,t)+\sum_{k=0}^{m} \sum\limits_{s=0}^k\sum\limits_{\begin{subarray}{c}
           t=0  \\
           s+t=m+\ell+1
      \end{subarray}}^{m-k}
            v_i^{-2b(2m-2k+1-2t)}\phi(k,s,t)+\\
            &\sum_{k=0}^{m} \sum\limits_{s=0}^{{{k-1}}}\sum\limits_{
            \begin{subarray}{c}
           t=0  \\
           s-t=\ell+2k-m
      \end{subarray}}^{m-k}
            v_i^{2b(2m-2k+1-2t)}\phi(k,s,t)+\sum_{k=0}^{m} \sum\limits_{s=0}^{{k-1}}\sum\limits_{
            \begin{subarray}{c}
           t=0  \\
           t-s=m+\ell+1-2k
      \end{subarray}}^{m-k}
           v_i^{-2b(2m-2k+1-2t)}\phi(k,s,t)\\
            =
            &\sum_{k=0}^{m} \sum\limits_{s=0}^k\sum\limits_{
            \begin{subarray}{c}
           t=0  \\
           t+s=m-\ell
      \end{subarray} }^{m-k}
            v_i^{2b(2k-2s)}\phi(k,s,t)+\sum_{k=0}^{m} \sum\limits_{s=0}^k\sum
            \limits_{
            \begin{subarray}{c}
           t=0  \\
           s+t=\ell+m+1
      \end{subarray}}^{m-k}
             v_i^{-2b(2k-2s)}\phi(k,s,t)+\\
            &\sum_{k=0}^{m} \sum\limits_{s=0}^{k-1}\sum\limits_{
            \begin{subarray}{c}
           t=0  \\
           s-t=\ell-m+2k
      \end{subarray} }^{{m-k}}
             v_i^{-2b(2k-2s)}\phi(k,s,t)+\sum_{k=0}^{m} \sum\limits_{s=0}^{k-1}\sum\limits_{
             \begin{subarray}{c}
           t=0  \\
           t-s=m+\ell-2k+1
      \end{subarray} }^{{m-k}}
            v_i^{2b(2k-2s)}\phi(k,s,t).
    \end{align*}
By symmetry, we only need to consider $0\leq \ell\leq m$. In this case, the second terms on both sides equal to zero. Thus the above equation can be reformulated as
\begin{align*}
     &\sum_{k=0}^{m} \sum\limits_{s=k-\ell}^{k}
            v_i^{2b(2\ell-2k+2s+1)}   \phi(k,s,m-\ell-s)
            +\sum_{k=0}^{m} \sum\limits_{s=0}^{{{k-1}}}
            v_i^{2b(2k+2\ell-2s+1)}\phi(k,s,m+1+\ell-s)\\
           &+\sum_{k=\ell+1}^{m} \sum\limits_{s=0}^{k-\ell-1}
             v_i^{-2b(2k-2\ell-2s-1)} \phi(k,s,m-\ell-s)\\
      =   &\sum_{k=0   }^{m} \sum\limits_{s=k-\ell}^k
            v_i^{2b(2k-2s)}\phi(k,s,m-\ell-s)
            +\sum_{k=0}^{m} \sum\limits_{s=0}^{k-1}
            v_i^{-2b(2k-2s)}\phi(k,s,m+1+\ell-s)\\
            &+\sum_{k=\ell+1}^{m} \sum\limits_{s=0}^{k-\ell-1}
            v_i^{2b(2k-2s)}\phi(k,s,m-\ell-s).
    \end{align*}
    That is,

\begin{align*}
     &\sum_{k=0}^{m} \sum\limits_{s=0}^{k}
            v_i^{2b(2\ell-2k+2s+1)}   \phi(k,s,m-\ell-s)
           +\sum_{k=0}^{m} \sum\limits_{s=0}^{{{k-1}}}
            v_i^{2b(2k+2\ell-2s+1)}\phi(k,s,m+1+\ell-s)
              \\
      &=   \sum_{k=0   }^{m} \sum\limits_{s=0}^k
            v_i^{2b(2k-2s)}\phi(k,s,m-\ell-s)
            +\sum_{k=0}^{m} \sum\limits_{s=0}^{k-1}
            v_i^{-2b(2k-2s)}\phi(k,s,m+1+\ell-s).
    \end{align*}
    Equivalently,
\begin{align*}
     &\sum_{k=0}^{m} \sum\limits_{s=0}^{k}
            (v_i^{2b(2\ell+1-2k+2s)}-v_i^{2b(2k-2s)})  \phi(k,s,m-\ell-s)
                         \\
      =   &\sum_{k=0}^{m} \sum\limits_{s=0}^{k-1}
            (v_i^{-2b(2k-2s)}-v_i^{2b(2\ell+1+2k-2s)})\phi(k,s,m+1+\ell-s).
    \end{align*}
Via replacing $s$ by ${2}k-s$ on the right-hand side, we obtain
    \begin{align*}
     &\sum_{k=0}^{m} \sum\limits_{s=0}^{k}
            (v_i^{2b(2\ell+1-2k+2s)}-v_i^{2b(2k-2s)})\phi(k,s,m-\ell-s)\\
      &=\sum_{k=0}^{m} \sum\limits_{s=k+1}^{2k}
             (v_i^{2b(2k-2s)}-v_i^{2b(2\ell+1-2k+2s)})\phi(k,s,m-\ell-s).
    \end{align*}
That is,
\begin{align} \label{eqn:binorm-form}
     \sum_{k=0}^{m} \sum\limits_{s=0}^{2k}
          (v_i^{2b(2\ell+1-2k+2s)}-v_i^{2b(2k-2s)})\phi(k,s,m-\ell-s)=0.
    \end{align}
Note that
$$v_i^{2b(2\ell+1-2k+2s)}-v_i^{2b(2k-2s)}=v_i^{b(2\ell+1)}(v_i-v_i^{-1}) [b(2\ell+1-4k+4s)]_{v_i}.$$
Therefore, we have established the following.
\begin{proposition}
\label{prop:reduce-even-even}
    Assume $c_{ij}=-2m$ is even and $\ov{p}_i=0$.
    Then \eqref{yzgx} holds if and only if for any $0\leq |b|,\ell\leq m$,
\begin{align}\label{ou0}
     &\sum_{k=0}^{m} \sum\limits_{s=0}^{2k}
           [b(2\ell+1-4k+4s)]_{v_i}\qbinom{2m+1}{2k}_{v_i}\Big(\qbinom{2k-1}{s}_{v_i^2}+\qbinom{2k-1}{s-1}_{v_i^2}\Big)\qbinom{2m-2k+1}{m-\ell-s}_{v_i^2}=0.
    \end{align}
\end{proposition}

\subsubsection{Even $c_{ij}$ and $\ov{p}_i=1$}
In this case, we need to prove
    \begin{align*}
			\sum_{n=0}^{2m+1} (-1)^n  y_{i,\ov{1}}^{(n)}\cdot X_j\cdot y_{i,\ov{1}}^{(2m+1-n)} =0,
	\end{align*}
    which is equivalent to
     \begin{align*}
			\sum_{k=0}^{m} y_{i,\ov{1}}^{(2k)}\cdot X_j \cdot y_{i,\ov{1}}^{(2m+1-2k)} =\sum_{k=0}^{m} y_{i,\ov{1}}^{(2k+1)}\cdot X_j \cdot y_{i,\ov{1}}^{(2m-2k)}.
	\end{align*}
    That is,
     \begin{align*}
			&\sum_{k=0}^{m}
            \frac{1}{[2k]^!_{v_i}}\sum\limits_{t=0}^k\qbinom{2k}{t}_{v_i^2}y_{i,2k-2t}u_i^t\cdot X_j \\
            &\quad\quad\cdot\frac{1}{[2m-2k+1]^!_{v_i}}\sum\limits_{s=0}^{m-k}(\qbinom{2m-2k}{s}_{v_i^2}+\qbinom{2m-2k}{s-1}_{v_i^2})y_{i,2m-2k+1-2s}u_i^s\\
            =&\sum_{k=0}^{m}  \frac{1}{[2k+1]^!_{v_i}}\sum\limits_{s=0}^k(\qbinom{2k}{s}_{v_i^2}+\qbinom{2k}{s-1}_{v_i^2})y_{i,2k+1-2s}u_i^s\cdot X_j \\
            & \quad\quad\cdot\frac{1}{[2m-2k]^!_{v_i}}\sum\limits_{t=0}^{m-k}\qbinom{2m-2k}{t}_{v_i^2}y_{i,2m-2k-2t}u_i^t.
	\end{align*}
Via replacing $k$ by $m-k$ on the left-hand side, it suffices to show
\begin{align}\label{eq:even and one}
			\sum_{k=0}^{m} \sum\limits_{s=0}^k\sum\limits_{t=0}^{m-k}
            \eta(k,s,t)  y_{i,2m-2k-2t}
            \cdot X_j\cdot y_{i,2k-2s+1} u_i^{s+t}\\\notag
            =\sum_{k=0}^{m}
    \sum\limits_{s=0}^k \sum\limits_{t=0}^{m-k}
    \eta(k,s,t)
    y_{i,2k-2s+1}
    \cdot X_j\cdot  y_{i,2m-2k-2t} u_i^{s+t},
	\end{align}
where
    $$\eta(k,s,t):=\frac{1}{[2k+1]^!_{v_i}}\frac{1}{[2m-2k]^!_{v_i}}\big(\qbinom{2k}{s}_{v_i^2}+\qbinom{2k}{s-1}_{v_i^2}\big)
            \qbinom{2m-2k}{t}_{v_i^2}.$$
   By Lemma \ref{yxy},
it suffices to show that the  coefficients of
$X_i^{2\ell+1}\cdot X_j\cdot u_i^{m-\ell}=X_i^{m+\ell+1}\cdot Y_i^{m-\ell}\cdot X_j$
on both sides of \eqref{eq:even and one} coincide for any $-m\leq \ell\leq m$. That is,
      \begin{align*}
      &\sum_{k=0}^{m} \sum\limits_{s=0}^k\sum\limits_{s+t=m-\ell,t=0}^{m-k}v_i^{2b(2k-2s+1)}\eta(k,s,t)
      +\sum_{k=0}^{m} \sum\limits_{s=0}^k\sum\limits_{t+s=m+\ell+1,t=0}^{m-k}v_i^{-2b(2k-2s+1)}\eta(k,s,t)\\
      +&\sum_{k=0}^{m} \sum\limits_{s=0}^{{{k}}}\sum\limits_{t-s=\ell+m-2k,t=0}^{m-k-1}v_i^{2b(2k-2s+1)}\eta(k,s,t)
      +\sum_{k=0}^{m} \sum\limits_{s=0}^{{k}}\sum\limits_{s-t=\ell-m+2k+1,t=0}^{m-k-1}v_i^{-2b(2k-2s+1)}\eta(k,s,t)\\
            =
      &\sum_{k=0}^{m} \sum\limits_{s=0}^k\sum\limits_{t+s=m-\ell,t=0}^{m-k}v_i^{2b(2m-2k-2t)}\eta(k,s,t)
      +\sum_{k=0}^{m} \sum\limits_{s=0}^k\sum\limits_{s+t=\ell+m+1,t=0}^{m-k}v_i^{-2b(2m-2k-2t)}\eta(k,s,t)\\
      +&\sum_{k=0}^{m} \sum\limits_{s=0}^k\sum\limits_{s-t=\ell-m+2k+1,t=0}^{{m-k-1}}v_i^{2b(2m-2k-2t)}\eta(k,s,t)
      +\sum_{k=0}^{m} \sum\limits_{s=0}^k\sum\limits_{t-s=\ell+m-2k,t=0}^{{m-k-1}}v_i^{2b(2k+2t-2m)}\eta(k,s,t).
    \end{align*}
By symmetry, we only need to consider $0\leq \ell\leq m$. In this case, the second terms on both sides equal to zero. Thus the above equation can be reformulated as
\begin{align*}&\sum_{k=0}^{m} \sum\limits_{s=k-\ell}^k
            v_i^{2b(2k-2s+1)}\eta(k,s,m-\ell-s)+\sum_{k=0}^{m} \sum\limits_{s=2k-m-\ell}^{{{k-\ell-1}}}
            v_i^{2b(2k-2s+1)}\eta(k,s,m-\ell-s)\\
              +&\sum_{k=0}^{m} \sum\limits_{s=\ell-m+2k+1}^{{k}}
            v_i^{-2b(2k-2s+1)}\eta(k,s,m+\ell+1-s)\\
            =
            &\sum_{k=0}^{m} \sum\limits_{s=k-\ell}^k
           v_i^{2b(2\ell-2k+2s)}\eta(k,s,m-\ell-s)
            +\sum_{k=0}^{m} \sum\limits_{s=\ell-m+2k+1}^k
            v_i^{2b(2k-2s+2\ell+2)}\eta(k,s,m+\ell+1-s)\\
            +&\sum_{k=0}^{m} \sum\limits_{s=2k-m-\ell}^{k-\ell-1}
            v_i^{-2b(2k-2s-2\ell)}\eta(k,s,m-\ell-s).
    \end{align*}
   Equivalently, we need to prove
    \begin{align*}&\sum_{k=0}^{m} \sum\limits_{s=2k-m-\ell}^{k}
            v_i^{2b(2k-2s+1)}\eta(k,s,m-\ell-s)
              +\sum_{k=0}^{m} \sum\limits_{s=\ell-m+2k+1}^{{k}}
            v_i^{2b(2s-2k-1)}\eta(k,s,m+\ell+1-s)\\
            &=
            \sum_{k=0}^{m} \sum\limits_{s=2k-m-\ell}^{k}
            v_i^{2b(2\ell-2k+2s)}\eta(k,s,m-\ell-s)
            \\&\quad+\sum_{k=0}^{m} \sum\limits_{s=\ell-m+2k+1}^k
            v_i^{2b(2k-2s+2\ell+2)}\eta(k,s,m+\ell+1-s),
    \end{align*}
which is reorganized to be
\begin{align*}&\sum_{k=0}^{m} \sum\limits_{s=0}^k
            \big(v_i^{2b(2k-2s+1)}-v_i^{2b(2\ell-2k+2s)}\big)\eta(k,s,m-\ell-s)\\
            &=\sum_{k=0}^{m} \sum\limits_{s=0}^k
            \big(v_i^{2b(2k-2s+2\ell+2)}-v_i^{-2b(2k-2s+1)}\big)\eta(k,s,m+\ell+1-s).
    \end{align*}
Replacing $s$ by ${2}k-s+1$ on the right-hand side, we obtain
         \begin{align*}
    \sum_{k=0}^{m} \sum\limits_{s=0}^{2k+1}
            (v_i^{2b(2k-2s+1)}-v_i^{2b(2\ell-2k+2s)})\eta(k,s,m-\ell-s)=0.
    \end{align*}
Note that
$$v_i^{2b(2\ell-2k+2s)}-v_i^{2b(2k-2s+1)}=v_i^{b(2\ell+1)}(v_i-v_i^{-1}) [b(2\ell-1-4k+4s)]_{v_i}.$$
Therefore, we have established the following.
\begin{proposition}\label{prop:even and p=1}
    Assume $c_{ij}=-2m$ is even and $\ov{p}_i=1$.
    Then \eqref{yzgx} holds if and only if for any $0\leq |b|,\ell\leq m$,
\begin{align}\label{evenp=1}
       \begin{split}
     \sum_{k=0}^{m} \sum\limits_{s=0}^{2k+1}
            [b(2\ell-1-4k+4s)]_{v_i}\qbinom{2m+1}{2k+1}_{v_i}\Big(\qbinom{2k}{s}_{v_i^2}+\qbinom{2k}{s-1}_{v_i^2}\Big)\qbinom{2m-2k}{m-\ell-s}_{v_i^2}=0.
             \end{split}
    \end{align}
\end{proposition}

\subsection{Odd $c_{ij}$}
\label{subsec:odd-even}

In this subsection, let us consider the proof of (\ref{yzgx}) for the case that $c_{ij}$ is odd. Assume $c_{ij}=-(2m+1)$, $\n_{ij}=-ad_i$ and $\n_{ji}=-(a+2b+1)d_i$ for some non-negative integers $a,m$, and $-(m+1)\leq b\leq m$. Then $m=a+b$.

\subsubsection{Odd $c_{ij}$ and $\ov{p}_i=0$}
\label{subsubsec:odd-even}

In this case, we need to prove
\begin{align*}
\sum_{n=0}^{2m+2} (-1)^n  y_{i,\ov{0}}^{(n)}\cdot X_j\cdot y_{i,\ov{1}}^{(2m+2-n)} =0,
\end{align*}
which is equivalent to
\begin{align*}
\sum_{k=0}^{m+1} y_{i,\ov{0}}^{(2k)}\cdot X_j\cdot y_{i,\ov{1}}^{(2m-2k+2)} =\sum_{k=0}^{m} y_{i,\ov{0}}^{(2k+1)}\cdot X_j \cdot y_{i,\ov{1}}^{(2m-2k+1)}.
\end{align*}
That is,
\begin{align*}
&\sum_{k=0}^{m+1}\frac{1}{[2k]^!_{v_i}}\sum\limits_{s=0}^k\big(\qbinom{2k-1}{s}_{v_i^2}+\qbinom{2k-1}{s-1}_{v_i^2}\big)y_{i,2k-2s}u_i^s
\cdot X_j \\
&\quad\quad\cdot\frac{1}{[2m-2k+2]^!_{v_i}}\sum\limits_{t=0}^{m-k+1}\qbinom{2m-2k+2}{t}_{v_i^2}y_{i,2m-2k+2-2t}u_i^t\\
&=\sum_{k=0}^{m}  \frac{1}{[2k+1]^!_{v_i}}\sum\limits_{t=0}^k\qbinom{2k+1}{t}_{v_i^2}y_{i,2k+1-2t}u_i^t
\cdot X_j \\
&\quad\quad\cdot\frac{1}{[2m-2k+1]^!_{v_i}}\sum\limits_{s=0}^{m-k}\big(\qbinom{2m-2k}{s}_{v_i^2}+\qbinom{2m-2k}{s-1}_{v_i^2}\big)y_{i,2m-2k+1-2s}u_i^s.
\end{align*}
Via replacing $k$ by $m-k$ on the right-hand side, it suffices to show
 \begin{align}\label{four left and right}
&\sum_{k=0}^{m+1} \sum\limits_{s=0}^k\sum\limits_{t=0}^{m-k+1}\xi(k,s,t) y_{i,2k-2s}\cdot X_j\cdot y_{i,2m-2k+2-2t} u_i^{s+t}\\\notag
&=\sum_{k=0}^{m}\sum\limits_{s=0}^k \sum\limits_{t=0}^{m-k}\psi(k,s,t) y_{i,2m-2k+1-2t}\cdot X_j \cdot  y_{i,2k+1-2s} u_i^{s+t},
\end{align}
where
\begin{align*}&\xi(k,s,t):=\frac{1}{[2k]^!_{v_i}}\frac{1}{[2m-2k+2]^!_{v_i}}\big(\qbinom{2k-1}{s}_{v_i^2}+\qbinom{2k-1}{s-1}_{v_i^2}\big)
            \qbinom{2m-2k+2}{t}_{v_i^2}\\
            &\psi(k,s,t):=\frac{1}{[2k+1]^!_{v_i}}  \frac{1}{[2m-2k+1]^!_{v_i}}  \big(\qbinom{2k}{s}_{v_i^2}+\qbinom{2k}{s-1}_{v_i^2}\big)\qbinom{2m-2k+1}{t}_{v_i^2}.\end{align*}
By Lemma \ref{yxy},
it suffices to show that the coefficients of $X_i^{2\ell+2}\cdot X_j\cdot u_i^{m-\ell}=X_i^{m+\ell+2}\cdot Y_i^{m-\ell}\cdot X_j$ on both sides of \eqref{four left and right} are identical for any $-m\leq \ell\leq m$. That is,
\begin{align*}
&\sum_{k=0}^{m+1} \sum\limits_{s=0}^k\sum\limits_{\begin{subarray}{c}t=0\\t+s=m-\ell\end{subarray}}^{m-k+1}v_i^{(2b+1)(2m-2k-2t+2)}\xi(k,s,t)
+\sum_{k=0}^{m+1} \sum\limits_{s=0}^k\sum\limits_{\begin{subarray}{c}t=0\\ s+t=m+\ell+2\end{subarray}}^{m-k+1}v_i^{-(2b+1)(2m-2k-2t+2)}
\xi(k,s,t)
\\
&+\sum_{k=0}^{m+1} \sum\limits_{s=0}^{{{k-1}}}\sum\limits_{\begin{subarray}{c}t=0\\ s-t=\ell+2k-m\end{subarray}}^{{m-k}}
v_i^{(2b+1)(2m-2k-2t+2)}\xi(k,s,t)
\\
&
+\sum_{k=0}^{m+1} \sum\limits_{s=0}^{{k-1}}\sum\limits_{\begin{subarray}{c}t=0\\t-s=m+\ell-2k+2\end{subarray}}^{{m-k}}
v_i^{-(2b+1)(2m-2k-2t+2)}\xi(k,s,t)\\
&=\sum_{k=0}^{m} \sum\limits_{s=0}^k\sum\limits_{\begin{subarray}{c}t=0\\t-s=\ell+m-2k+1\end{subarray}}^{{m-k}}v_i^{(2b+1)(2k-2s+1)}\psi(k,s,t)
+\sum_{k=0}^{m} \sum\limits_{s=0}^k\sum\limits_{\begin{subarray}{c}t=0\\s+t=\ell+m+2\end{subarray}}^{m-k}v_i^{-(2b+1)(2k-2s+1)}\psi(k,s,t)\\
&+
\sum_{k=0}^{m} \sum\limits_{s=0}^k\sum\limits_{\begin{subarray}{c}t=0\\ t+s=m-\ell\end{subarray}}^{m-k}v_i^{(2b+1)(2k-2s+1)}\psi(k,s,t)+\sum_{k=0}^{m} \sum\limits_{s=0}^k\sum\limits_{\begin{subarray}{c}t=0\\s-t=\ell-m+2k+1\end{subarray}}^{{m-k}}v_i^{(2b+1)(2s-2k-1)}\psi(k,s,t).
\end{align*}
By symmetry, we only need to consider $0\leq \ell\leq m$. In this case, the second terms on both sides equal to zero. Thus the above equation can be reformulated as
\begin{align*}
&\sum_{k=0}^{m+1} \sum\limits_{s=k-\ell-1}^k
v_i^{(2b+1)(2\ell-2k+2s+2)}\xi(k,s,m-\ell-s)+\sum_{k=0}^{m+1} \sum\limits_{s=\ell+2k-m}^{{{k-1}}}
v_i^{(2b+1)(2\ell+2k-2s+2)}
\\
&\qquad\cdot \xi(k,s,m+\ell+2-s)
+\sum_{k=0}^{m+1} \sum\limits_{s=2k-m-\ell-2}^{{k-\ell-2}}
v_i^{-(2b+1)(2k-2\ell-2s-2)}\xi(k,s,m-\ell-s)\\
=
&\sum_{k=0}^{m} \sum\limits_{s=k-\ell}^kv_i^{(2b+1)(2k-2s+1)}\psi(k,s,m-\ell-s)+\sum_{k=0}^{m} \sum\limits_{s=2k-\ell-m-1}^{k-\ell-1}
v_i^{(2b+1)(2k-2s+1)}
\\
&\qquad\cdot\psi(k,s,m-\ell-s)
+\sum_{k=0}^{m} \sum\limits_{s=2k+\ell-m+1}^kv_i^{-(2b+1)(2k-2s+1)}\psi(k,s,m+\ell+2-s).
\end{align*}
Equivalently, we need to prove
\begin{equation}\label{odd01}
\begin{split}
&\sum_{k=0}^{m+1} \sum\limits_{s=0}^k
v_i^{(2b+1)(2\ell-2k+2s+2)}\xi(k,s,m-\ell-s)
+\sum_{k=0}^{m+1} \sum\limits_{s=0}^{{{k-1}}}
v_i^{(2b+1)(2\ell+2k-2s+2)}\xi(k,s,m+\ell+2-s)\\
&=\sum_{k=0}^{m} \sum\limits_{s=0}^k
v_i^{(2b+1)(2k-2s+1)}\psi(k,s,m-\ell-s)
+\sum_{k=0}^{m} \sum\limits_{s=0}^k
v_i^{-(2b+1)(2k-2s+1)}\psi(k,s,m+\ell+2-s).
\end{split}\end{equation}
Replacing $s$ by $2k-s$ in the second term on the left-hand side of (\ref{odd01}), we obatin
\begin{align}\label{lhsodd0}
{\rm LHS}\eqref{odd01}=\sum_{k=0}^{m+1} \sum\limits_{s=0}^{2k}v_i^{(2b+1)(2\ell-2k+2s+2)}\xi(k,s,m-\ell-s).
\end{align}
Replacing $s$ by $2k-s+1$ in the second term on the right-hand side of (\ref{odd01}), we obtain
\begin{align}\label{rhsodd0}
{\rm RHS}\eqref{odd01}=\sum_{k=0}^{m} \sum\limits_{s=0}^{2k+1}v_i^{(2b+1)(2k-2s+1)}\psi(k,s,m-\ell-s).
\end{align}
By multiplying $[2m+2]^!_{v_i}$ in (\ref{lhsodd0}) and (\ref{rhsodd0}), we need to prove
\begin{align*}
& \sum_{k=0}^{m+1} \sum\limits_{s=0}^{2k}v_i^{(2b+1)(2\ell-2k+2s+2)}
\qbinom{2m+2}{2k}_{v_i}\Big(\qbinom{2k-1}{s}_{v_i^2}+\qbinom{2k-1}{s-1}_{v_i^2}\Big)\qbinom{2m-2k+2}{m-\ell-s}_{v_i^2}\\
&=\sum_{k=0}^{m} \sum\limits_{s=0}^{2k+1}v_i^{(2b+1)(2k-2s+1)}\qbinom{2m+2}{2k+1}_{v_i}\Big(\qbinom{2k}{s}_{v_i^2}+\qbinom{2k}{s-1}_{v_i^2}\Big)\qbinom{2m-2k+1}{m-\ell-s}_{v_i^2},
\end{align*}
which can be reformulated as follows via replacing $v_i$ by $v_i^{-1}$:
\begin{equation}\label{oddbdym-pre}
\begin{split}
& \sum_{k,s\geq0}v_i^{(2b+1)(2\ell-2k+2s+1)}\qbinom{2m+2}{2k+1}_{v_i}\Big(\qbinom{2k}{s}_{v_i^2}+\qbinom{2k}{s-1}_{v_i^2}\Big)\qbinom{2m-2k+1}{m-\ell-s}_{v_i^2}\\
&=\sum_{k,s\geq0}v_i^{(2b+1)(2k-2s)}\qbinom{2m+2}{2k}_{v_i}\Big(\qbinom{2k-1}{s}_{v_i^2}+\qbinom{2k-1}{s-1}_{v_i^2}\Big)\qbinom{2m-2k+2}{m-\ell-s}_{v_i^2}.
 \end{split}\end{equation}

\subsubsection{Odd $c_{ij}$ and $\ov{p}_i=1$}
\label{subsec:odd-odd}

In this case, we need to prove
    \begin{align}		\label{eq:odd-odd}
            \sum_{n=0}^{2m+2} (-1)^n  y_{i,\ov{1}}^{(n)}\cdot X_j \cdot y_{i,\ov{0}}^{(2m+2-n)} =0,
	\end{align}
    which is equivalent to
     \begin{align*}
			\sum_{k=0}^{m+1} y_{i,\ov{1}}^{(2k)}\cdot X_j \cdot y_{i,\ov{0}}^{(2m-2k+2)} =\sum_{k=0}^{m} y_{i,\ov{1}}^{(2k+1)}\cdot X_j\cdot y_{i,\ov{0}}^{(2m-2k+1)}.
	\end{align*}
    That is,
     \begin{align*}
			&\sum_{k=0}^{m+1}
            \frac{1}{[2k]^!_{v_i}}\sum\limits_{t=0}^k\qbinom{2k}{t}_{v_i^2}y_{i,2k-2t}u_i^t\cdot X_j\\
            &\quad\quad\cdot \frac{1}{[2m-2k+2]^!_{v_i}}
            \sum\limits_{s=0}^{m-k+1}(\qbinom{2m-2k+1}{s}_{v_i^2}+\qbinom{2m-2k+1}{s-1}_{v_i^2})y_{i,2m-2k+2-2s}u_i^s\\
            &=\sum_{k=0}^{m}  \frac{1}{[2k+1]^!_{v_i}}
    \sum\limits_{s=0}^k(\qbinom{2k}{s}_{v_i^2}+\qbinom{2k}{s-1}_{v_i^2})y_{i,2k+1-2s}u_i^s\cdot X_j\\
    &\quad\quad\cdot\frac{1}{[2m-2k+1]^!_{v_i}}\sum\limits_{t=0}^{m-k}\qbinom{2m-2k+1}{t}_{v_i^2}
     y_{i,2m-2k+1-2t}u_i^t.
	\end{align*}
Via replacing $k$ by $m-k+1$ on the left-hand side, it suffices to show
 \begin{align}\label{third left and right}
&\sum_{k=0}^{m+1} \sum\limits_{s=0}^k\sum\limits_{t=0}^{m-k+1}\xi(k,s,t) y_{i,2m-2k+2-2t} \cdot X_j\cdot y_{i,2k-2s}u_i^{s+t}\\\notag
=&\sum_{k=0}^{m}\sum\limits_{s=0}^k \sum\limits_{t=0}^{m-k}\psi(k,s,t) y_{i,2k+1-2s}\cdot X_j \cdot y_{i,2m-2k+1-2t} u_i^{s+t}.
\end{align}

By Lemma \ref{yxy},
it suffices to show that the coefficients of $X_i^{2\ell+2}\cdot X_j\cdot u_i^{m-\ell}=X_i^{m+\ell+2}\cdot Y_i^{m-\ell}\cdot X_j$ on both sides of \eqref{third left and right} are identical for any $-m\leq \ell\leq m$. That is,
\begin{align*}
&\sum_{k=0}^{m+1} \sum\limits_{s=0}^k\sum\limits_{\begin{subarray}{c}{t=0}\\{s+t=m-\ell}\end{subarray}}^{m-k+1}v_i^{(2b+1)(2k-2s)}\xi(k,s,t)
+\sum_{k=0}^{m+1} \sum\limits_{s=0}^k\sum\limits_{\begin{subarray}{c}t=0\\ t+s=m+\ell+2\end{subarray}}^{m-k+1}v_i^{-(2b+1)(2k-2s)}\xi(k,s,t)\\
&+\sum_{k=0}^{m+1} \sum\limits_{s=0}^{{{k-1}}}\sum\limits_{\begin{subarray}{c}
   t=0\\ t-s=m+\ell-2k+2
\end{subarray}}^{{m-k}}v_i^{(2b+1)(2k-2s)}\xi(k,s,t)
+\sum_{k=0}^{m+1} \sum\limits_{s=0}^{{k-1}}\sum\limits_{\begin{subarray}{c}t=0\\ s-t=\ell+2k-m\end{subarray}}^{{m-k}}v_i^{-(2b+1)(2k-2s)}\xi(k,s,t)\\
&=\sum_{k=0}^{m} \sum\limits_{s=0}^k\sum\limits_{\begin{subarray}{c}t=0\\
    t+s=m-\ell
\end{subarray}}^{m-k}v_i^{(2b+1)(2m-2k-2t+1)}\psi(k,s,t)
+\sum_{k=0}^{m} \sum\limits_{s=0}^k\sum\limits_{\begin{subarray}{c}
    t=0\\s+t=\ell+m+2
\end{subarray}}^{m-k}
v_i^{(2b+1)(2k+2t-2m-1)}\psi(k,s,t)\\
&+\sum_{k=0}^{m} \sum\limits_{s=0}^k\sum\limits_{\begin{subarray}{c} t=0\\s-t=\ell-m+2k+1\end{subarray}}^{{m-k}}
v_i^{(2b+1)(2m-2k-2t+1)}\psi(k,s,t)
\\
&+\sum_{k=0}^{m} \sum\limits_{s=0}^k\sum\limits_{\begin{subarray}{c}t=0\\t-s=\ell+m-2k+1\end{subarray}}^{{m-k}}v_i^{-(2b+1)(2m-2k-2t+1)}\psi(k,s,t).
\end{align*}
By symmetry, we only need to consider $0\leq \ell\leq m$. In this case, the second terms on both sides equal to zero. Thus the above equation can be reformulated as
\begin{align*}
&\sum_{k=0}^{m+1} \sum\limits_{s=k-\ell-1}^kv_i^{(2b+1)(2k-2s)}\xi(k,s,m-\ell-s)
+\sum_{k=0}^{m+1} \sum\limits_{s=2k-m-\ell-2}^{{{k-\ell-2}}}v_i^{(2b+1)(2k-2s)}\xi(k,s,m-\ell-s)\\
&+\sum_{k=0}^{m+1} \sum\limits_{s=\ell+2k-m}^{{k-1}}v_i^{-(2b+1)(2k-2s)}\xi(k,s,m+\ell+2-s)\\
=&\sum_{k=0}^{m} \sum\limits_{s=k-\ell}^kv_i^{(2b+1)(2\ell-2k+2s+1)}\psi(k,s,m-\ell-s)
+\sum_{k=0}^{m} \sum\limits_{s=2k+\ell-m+1}^kv_i^{(2b+1)(2\ell+2k-2s+3)}
\\
&\qquad\cdot\psi(k,s,m+\ell+2-s)+\sum_{k=0}^{m} \sum\limits_{s=2k-\ell-m-1}^{k-\ell-1}v_i^{-(2b+1)(2k-2\ell-2s-1)}\psi(k,s,m-\ell-s).
\end{align*}
Note that when $s<2k-m-\ell-2$, we have $m-\ell-s>2m-2k+2$, hence $\xi(k,s,m-\ell-s)=0$; while $s<2k-m-\ell-1$ implies $m-\ell-s>2m-2k+1$, hence $\psi(k,s,m-\ell-s)=0$.
Therefore, the above equation can be reformulated as
\begin{align*}
&\sum_{k=0}^{m+1} \sum\limits_{s=0}^kv_i^{(2b+1)(2k-2s)}\xi(k,s,m-\ell-s)+
\sum_{k=0}^{m+1} \sum\limits_{s=0}^{{k-1}}v_i^{-(2b+1)(2k-2s)}\xi(k,s,m+\ell+2-s)\\
=
&\sum_{k=0}^{m} \sum\limits_{s=0}^kv_i^{(2b+1)(2\ell-2k+2s+1)}\psi(k,s,m-\ell-s)
\\
&\qquad\qquad+\sum_{k=0}^{m} \sum\limits_{s=0}^kv_i^{(2b+1)(2\ell+2k-2s+3)}\psi(k,s,m+\ell+2-s),
\end{align*}
which is equivalent to (\ref{odd01}) by taking $v_i$ to $v_i^{-1}$, and then equivalent to \eqref{oddbdym-pre}.

In summary of the computations carried out in  \S\ref{subsubsec:odd-even}--\S\ref{subsec:odd-odd}, we have established the following statement.
\begin{proposition}\label{prop4.2.1}
    Assume $c_{ij}=-(2m+1)$ is odd.
    Then \eqref{yzgx} holds if and only if for any $0\leq \ell\leq m$ and $-(m+1)\leq b\leq m$,
\begin{equation}\label{oddbdym}
\begin{split}
& \sum_{k,s\geq0}v_i^{(2b+1)(2\ell-2k+2s+1)}\qbinom{2m+2}{2k+1}_{v_i}\Big(\qbinom{2k}{s}_{v_i^2}+\qbinom{2k}{s-1}_{v_i^2}\Big)\qbinom{2m-2k+1}{m-\ell-s}_{v_i^2}\\
&=\sum_{k,s\geq0}v_i^{(2b+1)(2k-2s)}\qbinom{2m+2}{2k}_{v_i}\Big(\qbinom{2k-1}{s}_{v_i^2}+\qbinom{2k-1}{s-1}_{v_i^2}\Big)\qbinom{2m-2k+2}{m-\ell-s}_{v_i^2}.
 \end{split}\end{equation}
\end{proposition}

\section{A special case: without 2-cycles}
\label{sec:without2-loop}

The proof of the main theorem has been converted to the verification of the combinatorial identities \eqref{ou0}, \eqref{evenp=1} and \eqref{oddbdym} in Section \ref{sec:iSerre}.
It is rather challenging to prove these combinatorial identities 
directly. To circumvent this issue, in this section we will prove the main theorem for the quivers without $2$-cycles, i.e., either $\lr{e_i,e_j}_Q=0$ or
 $\lr{e_j,e_i}_Q=0$ for any $i\neq j\in\I$.
 \begin{proposition}
    \label{teshuqx}
    Let $Q$ be any valued quiver such that either $\lr{e_i,e_j}_Q=0$ or
 $\lr{e_j,e_i}_Q=0$ for any $i\neq j\in\I$.
    Then
there exists an algebra homomorphism $\varphi_{\mathbf{i}}: \tUi \rightarrow \mathcal{T}_\mathbf{i}(Q)$ defined as \eqref{eq:phi-form1}--\eqref{eq:phi-form2}.
\end{proposition}

\begin{proof}
By Propositions \ref{prop:reduce-even-even}, \ref{prop:even and p=1} and \ref{prop4.2.1}, it is enough to prove \eqref{ou0}, \eqref{evenp=1} and \eqref{oddbdym}  for the cases
\begin{align*}
    b=\begin{cases}
        \pm m, &\text{ if }2\mid c_{ij},
        \\
        m\text{ or }-(m+1), &\text{ if }2\nmid c_{ij}.
    \end{cases}
\end{align*}
Replacing $v_i$   by $v_i^{-1}$ in these identities, we only need to consider the case $b=m$. Then the proof is reduced to the verification of the following identities for any $0\leq \ell\leq m$ (by taking $v$ to $v_i$):
\begin{align}
\label{eqn:binorm-form11}
&\sum_{k=0}^{m} \sum\limits_{s=0}^{2k}[m(2\ell+1-4k+4s)]_v\qbinom{2m+1}{2k}_v\Big(\qbinom{2k-1}{s}_{v^2}+\qbinom{2k-1}{s-1}_{v^2}\Big)\qbinom{2m-2k+1}{m-\ell-s}_{v^2}=0,
\end{align}
\begin{align}\label{eq:evenp=1b=m}
       \begin{split}
     \sum_{k=0}^{m} \sum\limits_{s=0}^{2k+1}
            [m(2\ell-1-4k+4s)]_v\qbinom{2m+1}{2k+1}_v\Big(\qbinom{2k}{s}_{v^2}+\qbinom{2k}{s-1}_{v^2}\Big)\qbinom{2m-2k}{m-\ell-s}_{v^2}=0,
             \end{split}
    \end{align}
    \begin{align}\label{eq:oddp=1b=m}
       \begin{split}
     & \sum_{k,s\geq0}v^{(2m+1)(2\ell-2k+2s+1)}\qbinom{2m+2}{2k+1}_v\Big(\qbinom{2k}{s}_{v^2}+\qbinom{2k}{s-1}_{v^2}\Big)\qbinom{2m-2k+1}{m-\ell-s}_{v^2}\\
&=\sum_{k,s\geq0}v^{(2m+1)(2k-2s)}\qbinom{2m+2}{2k}_v\Big(\qbinom{2k-1}{s}_{v^2}+\qbinom{2k-1}{s-1}_{v^2}\Big)\qbinom{2m-2k+2}{m-\ell-s}_{v^2}.
             \end{split}
    \end{align}
   Thus, the proof is completed after we prove \eqref{eqn:binorm-form11}, \eqref{eq:evenp=1b=m} and \eqref{eq:oddp=1b=m} in the remainder of this section.
\end{proof}

\subsection{Proof of \eqref{eqn:binorm-form11}}
\label{subsec:proof-binorm-form11}

Set $t=(2k-1)(m-\ell-s)-s(2m-2k+1)$. Then by \eqref{2.2}, we have
\begin{align*}
\sum_{s=0}^{2k}v^{-2t}\qbinom{2k-1}{s}_{v^2}\qbinom{2m-2k+1}{m-\ell-s}_{v^2}
=\qbinom{2m}{m-\ell}_{v^2}.
\end{align*}
Note that $m(2\ell+1-4k+4s)+2t=2m\ell+2\ell-m-4k\ell$. Hence
\begin{align*}
&\sum_{s=0}^{2k}v^{m(2\ell+1-4k+4s)}\qbinom{2k-1}{s}_{v^2}\qbinom{2m-2k+1}{m-\ell-s}_{v^2}=v^{2m\ell+2\ell-m-4k\ell}\qbinom{2m}{m-\ell}_{v^2}.
\end{align*}
Replacing $v$ by $v^{-1}$, we have
\begin{align*}
\sum_{s=0}^{2k}v^{-m(2\ell+1-4k+4s)}\qbinom{2k-1}{s}_{v^2}\qbinom{2m-2k+1}{m-\ell-s}_{v^2}
=v^{-2m\ell-2\ell+m+4k\ell}\qbinom{2m}{m-\ell}_{v^2}.
\end{align*}
Similarly, we have
\begin{align*}
&\sum_{s=0}^{2k}v^{m(2\ell+1-4k+4s)}\qbinom{2k-1}{s-1}_{v^2}\qbinom{2m-2k+1}{m-\ell-s}_{v^2}
=v^{2m\ell+3m+2+2\ell-4k\ell-4k}\qbinom{2m}{m-\ell-1}_{v^2}.
\end{align*}
Replacing $v$ by $v^{-1}$, we have
\begin{align*}
\sum_{s=0}^{2k}v^{-m(2\ell+1-4k+4s)}\qbinom{2k-1}{s-1}_{v^2}\qbinom{2m-2k+1}{m-\ell-s}_{v^2}=v^{-2m\ell-3m-2-2\ell+4k\ell+4k}\qbinom{2m}{m-\ell-1}_{v^2}.
\end{align*}
Summarizing the above discussions, we have established that
 \eqref{eqn:binorm-form11} is equivalent to \begin{equation}\label{ou0bm}
\begin{split}
&\sum_{k=0}^m\Big(v^{2m\ell+2\ell-m-4k\ell}- v^{-(2m\ell+2\ell-m-4k\ell)}\Big)\qbinom{2m+1}{2k}_{v}\qbinom{2m}{m-\ell}_{v^2}\\
&+\sum_{k=0}^m \Big(v^{2m\ell+3m+2+2\ell-4k\ell-4k}-v^{-(2m\ell+3m+2+2\ell-4k\ell-4k)}\Big)\qbinom{2m+1}{2k}_{v}\qbinom{2m}{m-\ell-1}_{v^2}=0.
\end{split}
\end{equation}
Now the proof of \eqref{eqn:binorm-form11} is reduced to the following.
\begin{lemma}
The equation \eqref{ou0bm} holds
for any $0\leq \ell\leq m$.
\end{lemma}

\begin{proof}
If $\ell=m$, the equation (\ref{ou0bm}) follows from Lemma \ref{lem:L-1}.

For $0\leq\ell<m$, by multiplying $\frac{[m-\ell]_{v^2}^![m+\ell+1]_{v^2}^!(v^2-v^{-2})}{[2m]_{v^2}^!}$, it suffices to show
\begin{align*}
&\sum_{k=0}^m\Big(v^{2m\ell+2\ell-m-4k\ell}- v^{-(2m\ell+2\ell-m-4k\ell)}\Big)\Big(v^{2m+2\ell+2}-v^{-2m-2\ell-2}\Big)\qbinom{2m+1}{2k}_{v}+\\
&\sum_{k=0}^m\Big(v^{2m\ell+3m+2+2\ell-4k\ell-4k}-v^{-(2m\ell+3m+2+2\ell-4k\ell-4k)}\Big)\Big(v^{2m-2\ell}-v^{-2m+2\ell}\Big)\qbinom{2m+1}{2k}_{v}=0.\end{align*}
That is, we need to prove
\begin{align}
\notag
&\sum_{k=0}^m\qbinom{2m+1}{2k}_{v}\Big( v^{2m\ell+4\ell +m+2-4k\ell} -v^{2m\ell-3m-2-4k\ell} -v^{-(2m\ell-3m-2-4k\ell)} +v^{-(2m\ell+4\ell +m+2-4k\ell)}+\\\notag
&v^{2m\ell  +5m+2-4k\ell-4k} -v^{2m\ell+m+2+4\ell-4k\ell-4k} -v^{-(2m\ell+m+2+4\ell-4k\ell-4k)}+v^{-(2m\ell  +5m+2-4k\ell-4k)}\Big)\\\label{ou0bmt}
&=0.
\end{align}

By using Lemma \ref{lem:L-1}, we can replace $v^{-4k\ell}$ by $v^{4k\ell-(4m+2)\ell}$ in the first and second terms, and replace $v^{-4k\ell-4k}$ by $v^{4k\ell+4k-(4m+2)(\ell+1)}$ in the fifth and sixth terms in (\ref{ou0bmt}). Then it suffices to prove
\begin{align*}
   & \sum_{k=0}^m  v^{4k\ell-2m\ell} \Big(v^{2\ell+m+2}-v^{-2\ell-3m-2}-v^{3m+2}+v^{-4\ell-m-2}\Big)\qbinom{2m+1}{2k}_{v}
   \\
   &=-\sum_{k=0}^m  v^{4k(\ell+1)-2m\ell} \Big(v^{m-2\ell}-v^{-3m+2\ell}-v^{-m-2-4\ell}+v^{-5m-2}\Big)\qbinom{2m+1}{2k}_{v}.
\end{align*}
That is,
\begin{align*}
   & \sum_{k=0}^m  v^{4k\ell-2m\ell} \Big(v^{-\ell}(v^{2\ell+2m+2}-v^{-(2\ell+2m+2)})(v^{\ell-m}-v^{m-\ell})\Big)\qbinom{2m+1}{2k}_{v}
   \\
  & =\sum_{k=0}^m v^{4k(\ell+1)-2m\ell} \Big(v^{-2m-1-\ell} (v^{m+1+\ell}-v^{-(m+1+\ell)})(v^{2\ell-2m}-v^{2m-2\ell})\Big)\qbinom{2m+1}{2k}_{v}.
\end{align*}
Equivalently, we need to prove

\begin{align*}
\sum_{k=0}^m  v^{4k\ell} v^{2m+1} (v^{m+1+\ell}+v^{-(m+1+\ell)})\qbinom{2m+1}{2k}_{v}
=\sum_{k=0}^m v^{4k(\ell+1)} (v^{\ell-m}+v^{m-\ell})\qbinom{2m+1}{2k}_{v}.
\end{align*}
Then (\ref{ou0bm}) follows from Proposition \ref{propA.4}.
\end{proof}

\subsection{Proof of \eqref{eq:evenp=1b=m}}
Using similar arguments as given in \S\ref{subsec:proof-binorm-form11}, we need to prove
\begin{align*}
  &\sum_{k=0}^m\Big(v^{4m\ell-4k\ell}-v^{2m+4k\ell}\Big)\qbinom{2m+1}{2k+1}_{v}\qbinom{2m}{m-\ell}_{v^2}
   \\
   &+\sum_{k=0}^m \Big(v^{4m\ell+4m-4k\ell-4k}-v^{-2m+4k+4k\ell}\Big)\qbinom{2m+1}{2k+1}_{v}\qbinom{2m}{m-\ell-1}_{v^2}=0.
\end{align*}
Via replacing $k$ by $m-k$, we need to prove
\begin{align*}
  &\sum_{k=0}^m\Big(v^{4k\ell}-v^{2m+4m\ell-4k\ell}\Big)\qbinom{2m+1}{2k}_{v}\qbinom{2m}{m-\ell}_{v^2}
   \\
   &+\sum_{k=0}^m \Big(v^{4k\ell+4k}-v^{2m+4m\ell-4k-4k\ell}\Big)\qbinom{2m+1}{2k}_{v}\qbinom{2m}{m-\ell-1}_{v^2}=0.
\end{align*}
Then the proof is reduced to the following.
\begin{lemma}
For any $0\leq \ell\leq m$, we have
\begin{align*}
  &\sum_{k=0}^m\Big(v^{2m\ell+m-4k\ell}- v^{-(2m\ell+m-4k\ell)}\Big)\qbinom{2m+1}{2k}_{v}\qbinom{2m}{m-\ell}_{v^2}
   \\
   &+\sum_{k=0}^m \Big(v^{2m\ell+m-4k\ell-4k}-v^{-(2m\ell+m-4k\ell-4k)}\Big)\qbinom{2m+1}{2k}_{v}\qbinom{2m}{m-\ell-1}_{v^2}=0.
\end{align*}
\end{lemma}

\begin{proof}
If $\ell=m$, it follows from Lemma \ref{lem:L-1}.

For $0\leq\ell<m$, by multiplying $\frac{[m-\ell]_{v^2}^![m+\ell+1]_{v^2}^!(v^2-v^{-2})}{[2m]_{v^2}^!}$, it suffices to show
\begin{align*}
   & \sum_{k=0}^m\Big(v^{4k\ell}-v^{2m+4m\ell-4k\ell}\Big)\Big(v^{2m+2\ell+2}-v^{-2m-2\ell-2}\Big)\qbinom{2m+1}{2k}_{v}
    \\
    &+\sum_{k=0}^m\Big(v^{4k\ell+4k}-v^{2m+4m\ell-4k-4k\ell}\Big)\Big(v^{2m-2\ell}-v^{-2m+2\ell}\Big)\qbinom{2m+1}{2k}_{v}=0.\end{align*}
  That is, we need to prove
  \begin{align}
  \label{8zhishu}
      &\sum_{k=0}^m\Big( v^{2\ell +2m+2+4k\ell} -v^{-2\ell-2m-2+4k\ell} -v^{4m\ell+4m+2\ell+2-4k\ell} +v^{4m\ell-2\ell-2-4k\ell}
    \\\notag
    &+v^{2m-2\ell+4k\ell+4k} -v^{-2m+2\ell+4k\ell+4k} -v^{4m\ell+4m-2\ell-4k\ell-4k}+v^{4m\ell +2\ell-4k\ell-4k}\Big)\qbinom{2m+1}{2k}_{v}=0.
\end{align}

 By using Lemma \ref{lem:L-1}, in (\ref{8zhishu}) we can replace $v^{-4k\ell}$ by $v^{4k\ell-(4m+2)\ell}$ in the third and fourth terms, and replace $v^{-4k\ell-4k}$ by $v^{4k\ell+4k-(4m+2)(\ell+1)}$ in the seventh and eighth terms. Then it suffices to prove
\begin{align*}
   & \sum_{k=0}^m  v^{4k\ell} \Big(v^{2\ell+2m+2}-v^{-2\ell-2m-2}-v^{4m+2}+v^{-4\ell-2}\Big)\qbinom{2m+1}{2k}_{v}
   \\&=-\sum_{k=0}^m  v^{4k(\ell+1)} \Big(v^{2m-2\ell}-v^{-2m+2\ell}-v^{-4\ell-2}+v^{-4m-2}\Big)\qbinom{2m+1}{2k}_{v}.
\end{align*}
That is, we need to prove
\begin{align*}
   & \sum_{k=0}^m  v^{4k\ell} \Big(v^{-\ell+m}(v^{2\ell+2m+2}-v^{-(2\ell+2m+2)})(v^{\ell-m}-v^{m-\ell})\Big)\qbinom{2m+1}{2k}_{v}
   \\&=\sum_{k=0}^m v^{4k(\ell+1)} \Big(v^{-m-1-\ell} (v^{m+1+\ell}-v^{-(m+1+\ell)})(v^{2\ell-2m}-v^{2m-2\ell})\Big)\qbinom{2m+1}{2k}_{v},
\end{align*}
which can be reformulated to be
\begin{align*}
  \sum_{k=0}^m  v^{4k\ell} v^{2m+1} (v^{m+1+\ell}+v^{-(m+1+\ell)})\qbinom{2m+1}{2k}_{v}
   =\sum_{k=0}^m v^{4k(\ell+1)} (v^{\ell-m}+v^{m-\ell})\qbinom{2m+1}{2k}_{v}.
\end{align*}
Then the result follows from Proposition \ref{propA.4}.
\end{proof}

\subsection{Proof of \eqref{eq:oddp=1b=m}}
Using similar arguments as given in \S\ref{subsec:proof-binorm-form11}, the equation \eqref{eq:oddp=1b=m} is equivalent to
\begin{equation}\label{512bm}
\begin{split}
&\sum_{k\geq0}\Big(v^{4m\ell-4k\ell+2m-2k+2\ell+1}\qbinom{2m+1}{m-\ell}_{v^2}+v^{4m\ell-4k\ell+6m-6k+2\ell+3}\qbinom{2m+1}{m-\ell-1}_{v^2}\Big)\qbinom{2m+2}{2k+1}_{v}\\
&=\sum_{k\geq0}\Big(v^{2(m-\ell+2k\ell+k)}\qbinom{2m+1}{m-\ell}_{v^2}+v^{2(-m-\ell-2+2k\ell+3k)}\qbinom{2m+1}{m-\ell-1}_{v^2}\Big)\qbinom{2m+2}{2k}_{v}.
\end{split}
\end{equation}
Now the proof is reduced to the following.
\begin{lemma}
The equation \eqref{512bm}
holds for any $0\leq \ell\leq m$.
\end{lemma}

\begin{proof}
If $\ell=m$, it is equivalent to
$$\sum\limits_{k\geq0} v^{(2k+1)(2m+1)}\qbinom{2m+2}{2k+1}_{v}=\sum\limits_{k\geq0} v^{2k(2m+1)}\qbinom{2m+2}{2k}_{v},$$ where we have replaced $k$ by $m-k$ in the first sum in (\ref{512bm}). Then it follows from Lemma \ref{lem:qbinom1}.

For $0\leq\ell<m$, we need to prove
\begin{align}
\label{eq:b=m2-1}
\begin{split}
&\sum_{k\geq0} \Big(v^{4m\ell-4k\ell+2m-2k+2\ell+1}(v^{2(m+\ell+2)}-v^{-2(m+\ell+2)}) \\
&\qquad+v^{4m\ell-4k\ell+6m-6k+2\ell+3}(v^{2(m-\ell)}-v^{-2(m-\ell)})\Big) \qbinom{2m+2}{2k+1}_{v}\\
=&\sum_{k\geq0} \Big(v^{2(m-\ell+2k\ell+k)}(v^{2(m+\ell+2)}-v^{-2(m+\ell+2)}) \\
&\qquad+v^{2(-m-\ell-2+2k\ell+3k)}(v^{2(m-\ell)}-v^{-2(m-\ell)})\Big)\qbinom{2m+2}{2k}_{v}.
\end{split}
\end{align}
Via replacing $k$ by $m-k$, the left-hand side of \eqref{eq:b=m2-1} equals to
\begin{align*}
\sum_{k\geq0} \Big(v^{2m+4\ell+5+2k(2\ell+1)}-v^{-2m-3+2k(2\ell+1)}+v^{2m+3+2k(2\ell+3)}-v^{-2m+4\ell+3+2k(2\ell+3)}\Big) \qbinom{2m+2}{2k+1}_{v}.
\end{align*}
By using Lemma \ref{lem:L-2}, the right-hand side of \eqref{eq:b=m2-1} equals to
\begin{align*}
&\sum_{k\geq0} \Big(v^{4m+4+2k(2\ell+1)}-v^{-4\ell-4+2k(2\ell+1)}+v^{-4\ell-4+2k(2\ell+3)}-v^{-4m-4+2k(2\ell+3)}\Big)\qbinom{2m+2}{2k}_{v}\\
        &=\sum_{k\geq0} \Big(v^{4m+2\ell+5+2k(2\ell+1)}-v^{2k(2\ell+1)-2\ell-3}+v^{2k(2\ell+3)-2\ell-1}-v^{2k(2\ell+3)-4m-1+2\ell}\Big)\qbinom{2m+2}{2k+1}_{v}.
\end{align*}
Thus, we need to prove
\begin{align*}
      &\sum_{k\geq0}v^{2k(2\ell+3)} \Big(v^{2m+3}-v^{-2m+4\ell+3}
       -v^{-2\ell-1}+v^{-4m-1+2\ell}\Big) \qbinom{2m+2}{2k+1}_{v}\\
       &=\sum_{k\geq0} v^{2k(2\ell+1)}\Big(v^{4m+2\ell+5}-v^{-2\ell-3}-v^{2m+4\ell+5}+v^{-2m-3}
       \Big)\qbinom{2m+2}{2k+1}_{v}.
    \end{align*}
That is,
   \begin{align*}
       &\sum_{k\geq0}v^{2k(2\ell+3)} v^{-m+\ell+1}\Big((v^{m+\ell+2}-v^{-(m+\ell+2)})(v^{2(m-\ell)}-v^{-2(m-\ell)})
       \Big) \qbinom{2m+2}{2k+1}_{v}\\
       &=\sum_{k\geq0} v^{2k(2\ell+1)}v^{m+\ell+1}\Big((v^{2m+2\ell+4}-v^{-(2m+2\ell+4)})(v^{m-\ell}-v^{-(m-\ell)})
       \Big)\qbinom{2m+2}{2k+1}_{v}.
    \end{align*}
 Equivalently, we need to prove
    \begin{align*}
      \sum_{k\geq0}v^{2k(2\ell+3)}  (v^{m-\ell}+v^{-(m-\ell)})\qbinom{2m+2}{2k+1}_{v}
       =\sum_{k\geq0} v^{2k(2\ell+1)+2m}(v^{m+\ell+2}+v^{-(m+\ell+2)})\qbinom{2m+2}{2k+1}_{v},
    \end{align*}
    which follows from Proposition \ref{propA.5}.
\end{proof}

\section{Completing the proof of main Theorem}
\label{sec:proof-Main}

In this section, let us prove that $\varphi_{\mathbf{i}}$ is an algebra homomorphism for an arbitrary valued quiver $Q$ associated to $C$, and finish the proof of Theorem \ref{mainthm}.

Without loss of generality, let us assume that $\tUi$ is of rank two, i.e., the quiver $Q$ has only two vertices $i\neq j$, moreover, assume $\lr{e_i,e_j}_Q\geq \lr{e_j,e_i}_Q$. Let $\lr{e_i,e_j}_Q=-td_i=-sd_j$ for some non-negative integers $s,t$. 
Set $$C'=\begin{bmatrix}
    2&c_{ij}+2t
    \\
    c_{ji}+2s& 2
\end{bmatrix}.$$ Then $C'$ is also a symmetrizable GCM with symmetrizer ${\rm diag}(d_i,d_j)$.
Let us associate with $C'$ an acyclic valued quiver $Q'$ such that $$\lr{e_i,e_j}_{Q'}=0~~\text{and}~~\lr{e_j,e_i}_{Q'}=d_ic_{ij}+2td_i.$$

Let $\mathbf{i}=(i_1i_2\cdots i_r)$. For any $1\leq k,l\leq r$ with $i_l=i, i_k=j$, we have in $\mathcal{T}_\mathbf{i}(Q')$
\begin{align*}
x_lx_k=v^{\lr{e_j,e_i}_{Q'}}x_kx_l=v^{d_ic_{ij}+2td_i}x_kx_l.
\end{align*}
While in $\mathcal{T}_\mathbf{i}(Q)$,
\begin{align*}
x_lx_k=v^{\lr{e_j,e_i}_{Q}-\lr{e_i,e_j}_Q}x_kx_l=v^{(e_i,e_j)_Q-2\lr{e_i,e_j}_Q}x_kx_l=v^{d_ic_{ij}+2td_i}x_kx_l.
\end{align*}
Thus, it is easy to see that the map $\gamma: \mathcal{T}_\mathbf{i}(Q')\rightarrow \mathcal{T}_\mathbf{i}(Q)$ defined on generators by $x_k^{\pm}\mapsto x_k^{\pm}$ is an isomorphism of algebras.

Let ${}'\tUi$ be the universal $\imath$quantum group of split type associated to $C'$. Since $\lr{e_i,e_j}_{Q'}=0$, by Proposition \ref{teshuqx}, we have an algebra homomorphism $\varphi'_{\mathbf{i}}:{}'\tUi\rightarrow\mathcal{T}_\mathbf{i}(Q')$.
By Proposition \ref{prop:SLinduct}, we have the higher order $\imath$Serre relations in ${}'\tUi$
\begin{align*}
\sum_{s+s'=1-c'_{ij}+2t} (-1)^s B^{(s)}_{i, \overline{p_i}} B_j B_{i,\overline{c'_{ij}}+\overline{p}_i}^{(s')}=0\quad\text{and}\quad\sum_{s+s'=1-c'_{ji}+2s} (-1)^s B^{(s)}_{j, \overline{p_j}} B_i B_{j,\overline{c'_{ji}}+\overline{p}_j}^{(s')}=0,
\end{align*}
where $c'_{ij}=c_{ij}+2t$, $c'_{ji}=c_{ji}+2s$. It follows that the relations \eqref{iserrerelation} for $\tUi$ hold in ${}'\tUi$.  Thus, the map $\iota:\tUi\rightarrow{}'\tUi$ defined on generators by $B_i\mapsto B_i$ and $\tk_i\mapsto \tk_i$ is an algebra homomorphism.
Using the commutative diagram
$$\xymatrix{\tUi\ar[r]^-{\varphi_\mathbf{i}}\ar[d]_-{\iota}&\mathcal{T}_\mathbf{i}(Q)\\
{}'\tUi\ar[r]^-{\varphi'_{\mathbf{i}}}&\mathcal{T}_\mathbf{i}(Q')\ar[u]^-{\gamma}}$$
we obtain that $\varphi_\mathbf{i}$ is an algebra homomorphism. Therefore, the proof of Theorem \ref{mainthm} is completed.

As a corollary of Theorem \ref{mainthm} and Propositions \ref{prop3.5}, \ref{prop:reduce-even-even}--\ref{prop4.2.1}, we have the following.
\begin{proposition}
    \label{cor:identities}
The combinatorial identities \eqref{ou0}, \eqref{evenp=1} hold for $0\leq |b|,\ell\leq m$, and \eqref{oddbdym} holds for $-(m+1)\leq b\leq m, 0\leq \ell\leq m$.
\end{proposition}

\appendix

\section{Some combinatorial identities}
\label{sec:Comb-identities}

In this section, we provide several combinatorial identities which are used in Section \ref{sec:without2-loop}.

\begin{lemma}[\text{cf. \cite[1.3.1(c)]{Lus93}; see also \cite[Lemma 5.4]{LW20a}}]
\label{lem:qbinom1}
Given a positive integer $p$, for any $d\in \Z$, if $|d| \le p-1$ and $d\equiv p-1\pmod 2$, then
$$\sum_{k=0}^{p}(-1)^kv^{-kd}\qbinom{p}{k}=0.$$
\end{lemma}

    \begin{lemma}
    \label{lem:L-2}
       Given a positive integer $p$, for any $d\in \Z$, the following statements hold:
       \begin{itemize}
           \item [(1)] if $|d| \le2p$ and $2\mid d$, then
        \begin{align}\label{5.2}
            \sum_{k=0}^{p}v^{-2kd}\qbinom{2p+1}{2k}=\sum_{k=0}^{p} v^{-(2k+1)d}\qbinom{2p+1}{2k+1}.
        \end{align}
        \item [(2)] if $|d| \le2p+1$ and $2\nmid d$, then
        \begin{align}\label{lem:L-2}
            \sum_{k=0}^{p+1}v^{-2kd}\qbinom{2p+2}{2k}=\sum_{k=0}^{p} v^{-(2k+1)d}\qbinom{2p+2}{2k+1}.
        \end{align}
       \end{itemize}
    \end{lemma}
\begin{proof}
We only prove (1), since (2) can be proved similarly. By Lemma \ref{lem:qbinom1}, we have
$$\sum_{k=0}^{2p+1}(-1)^kv^{-kd}\qbinom{2p+1}{k}=0.$$
Then it follows that \eqref{5.2} holds.
\end{proof}

    \begin{lemma}
    \label{lem:L-1}
       Given a positive integer $p$, for any $d\in \Z$, if $|d| \le2p$ and $2\mid d$, then
        \begin{align}
        \sum_{k=0}^p v^{-2kd} \qbinom{2p+1}{2k}=\sum_{k=0}^p v^{2kd-(2p+1)d} \qbinom{2p+1}{2k}.
        \end{align}
    \end{lemma}
\begin{proof}
Via replacing $k$ by $p-k$, the right-hand side of \eqref{5.2} can be reformulated as
\begin{align*}
           \sum_{k=0}^{p}v^{-(2k+1)d}\qbinom{2p+1}{2k+1}=\sum_{k=0}^{p}v^{-(2(p-k)+1)d}\qbinom{2p+1}{2(p-k)+1}=\sum_{k=0}^{p}v^{2kd-(2p+1)d}\qbinom{2p+1}{2k}.
        \end{align*}
        Then we finish the proof.
\end{proof}

\begin{proposition}
\label{propA.4}
For any $0\leq \ell<m$, we have
\begin{align}
\label{eq:odd-even-reduced}
   & \sum_{k=0}^m v^{4k(\ell+1)} (v^{m-\ell}+v^{-(m-\ell)})\qbinom{2m+1}{2k}=\sum_{k=0}^m  v^{4k\ell+2m+1}(v^{m+1+\ell}+v^{-(m+1+\ell)})\qbinom{2m+1}{2k}
   \\\label{2m+1 and 2k+1}
     & \sum_{k=0}^m v^{4k(\ell+1)} (v^{m-\ell}+v^{-(m-\ell)})\qbinom{2m+1}{2k+1}=\sum_{k=0}^m  v^{4k\ell+2m-1}(v^{m+1+\ell}+v^{-(m+1+\ell)})\qbinom{2m+1}{2k+1}.
\end{align}
\end{proposition}

\begin{proof}
 By \eqref{fkhgs}, we have
  \begin{align*}
   {\rm LHS}\eqref{eq:odd-even-reduced}&=\sum_{k=0}^m v^{4k(\ell+1)}(v^{m-\ell}-v^{-(m-\ell)})(v^{-2k}\qbinom{2m}{2k}+v^{-2k+(2m+1)}\qbinom{2m}{2k-1})\\
         &=\sum_{k=0}^m v^{4k\ell+2k} (v^{m-\ell}-v^{-(m-\ell)})\qbinom{2m}{2k}+\sum_{k=0}^m v^{4k\ell+2k} (v^{m+\ell+1}+v^{3m+1-\ell})\qbinom{2m}{2k-1}
 \end{align*}
 and
 \begin{align*}
   {\rm RHS}\eqref{eq:odd-even-reduced}&=\sum_{k=0}^m v^{4k\ell} v^{2m+1}(v^{m+\ell+1}+v^{-(m+\ell+1)}) (v^{2k}\qbinom{2m}{2k}+v^{2k-(2m+1)}\qbinom{2m}{2k-1})\\
         &=\sum_{k=0}^m v^{4k\ell+2k} (v^{3m+2+\ell}+v^{m-\ell})\qbinom{2m}{2k}+\sum_{k=0}^m v^{4k\ell+2k} (v^{m+\ell+1}+v^{-(m+\ell+1)})\qbinom{2m}{2k-1}.
 \end{align*}
 Hence it suffices to prove
  \begin{align*}
   \sum_{k=0}^m v^{4k\ell+2k} (v^{3m+\ell+2}-v^{\ell-m})\qbinom{2m}{2k}
        = \sum_{k=0}^m v^{4k\ell+2k} (v^{3m+1-\ell}-v^{-(m+\ell+1)})\qbinom{2m}{2k-1}.
 \end{align*}
 That is,
  \begin{align*}
   \sum_{k=0}^m v^{2k(2\ell+1)}\qbinom{2m}{2k}
        = \sum_{k=0}^m v^{2k(2\ell+1)} v^{-(2\ell+1)}\qbinom{2m}{2k-1}.
 \end{align*}
 Equivalently, we need to prove
  \begin{align*}
   \sum_{k=0}^m v^{2k(2\ell+1)}\qbinom{2m}{2k}
        = \sum_{k=0}^m v^{(2k-1)(2\ell+1)}\qbinom{2m}{2k-1},
 \end{align*}
which can be reorganized to be
 \begin{align*}
   \sum_{k=0}^{2m}(-1)^k v^{k(2\ell+1)}\qbinom{2m}{k}
        = 0,
 \end{align*}
 and then it follows from Lemma \ref{lem:qbinom1}.

 The equation \eqref{2m+1 and 2k+1} follows from  \eqref{eq:odd-even-reduced} via replacing $k$ by $m-k$ and then replacing $v$ by $v^{-1}$.
\end{proof}

\begin{proposition}
\label{propA.5}
For any $0\leq \ell<m$, we have
 \begin{align}
      \label{eq:even-odd-reduced2}
      \sum_{k\geq0}v^{2k(2\ell+3)}  (v^{m-\ell}+v^{-(m-\ell)})\qbinom{2m+2}{2k+1}
       =\sum_{k\geq0} v^{2k(2\ell+1)+2m}(v^{m+\ell+2}+v^{-(m+\ell+2)})\qbinom{2m+2}{2k+1}.
    \end{align}\end{proposition}
\begin{proof}
By \eqref{fkhgs}, we have
\begin{align*}
    &{\rm LHS}\eqref{eq:even-odd-reduced2}\\&=\sum_{k\geq0}v^{2k(2\ell+3)}  (v^{m-\ell}+v^{-(m-\ell)})\qbinom{2m+2}{2k+1}\\
    &=\sum_{k\geq0}v^{2k(2\ell+3)}  (v^{m-\ell}+v^{-(m-\ell)})(v^{-(2k+1)}\qbinom{2m+1}{2k+1}+v^{-(2k+1)+2m+2}\qbinom{2m+1}{2k})\\
    &=\sum_{k\geq0} v^{4k(\ell+1)} (v^{m-\ell}+v^{-(m-\ell)})(v^{-1}\qbinom{2m+1}{2k+1}+v^{2m+1}\qbinom{2m+1}{2k})\\
    &=\sum_{k\geq0}  v^{4k\ell}(v^{m+\ell+1}+v^{-(m+\ell+1)})(v^{2m-2}\qbinom{2m+1}{2k+1}+v^{4m+2}\qbinom{2m+1}{2k})\\
\end{align*}
where the last equality follows from Proposition \ref{propA.4}.

On the other hand,
\begin{align*}
&{\rm RHS}\eqref{eq:even-odd-reduced2}\\&=\sum_{k\geq0} v^{2k(2\ell+1)+2m}(v^{m+\ell+2}+v^{-(m+\ell+2)})(v^{-(2k+1)}\qbinom{2m+1}{2k+1}+v^{-(2k+1)+2m+2}\qbinom{2m+1}{2k})\\
    &=\sum_{k\geq0} v^{4k\ell} (v^{m+\ell+2}+v^{-(m+\ell+2)})(v^{2m-1}\qbinom{2m+1}{2k+1}+v^{4m+1}\qbinom{2m+1}{2k}).
\end{align*}
Then it suffices to show
\begin{align*}
    &\sum_{k\geq0}  v^{4k\ell}\big(v^{2m-2}\cdot (v^{m+\ell+1}+v^{-(m+\ell+1)})-(v^{m+\ell+2}+v^{-(m+\ell+2)}) \cdot v^{2m-1}\big)\qbinom{2m+1}{2k+1}\\
    &=\sum_{k\geq0} v^{4k\ell} \big(v^{4m+1}\cdot(v^{m+\ell+2}+v^{-(m+\ell+2)})-(v^{m+\ell+1}+v^{-(m+\ell+1)})\cdot v^{4m+2}\big)\qbinom{2m+1}{2k},
\end{align*}
or equivalently,
\begin{align*}
    &\sum_{k\geq0} v^{4k\ell+2\ell}\qbinom{2m+1}{2k+1}=\sum_{k\geq0} v^{4k\ell} \qbinom{2m+1}{2k},
\end{align*}
which follows from \eqref{5.2} by taking $d=-2\ell$.
\end{proof}

\end{document}